\documentclass[10pt,a4paper]{article}

\usepackage{pstricks}
\usepackage{dcolumn}% Align table columns on decimal point
\usepackage{bm}% bold math

\usepackage{amsmath,amsfonts,amssymb,latexsym}
\usepackage{amsthm}
\usepackage{dsfont}	
\usepackage{subfigure}
\usepackage[T1]{fontenc}
\usepackage{amsthm}
\usepackage{changes}

\usepackage{subfigure}
\usepackage{caption}
\usepackage[pdftex]{graphicx}
\usepackage{graphics}
\usepackage{epstopdf}

\usepackage{amsmath,amsfonts,amssymb,latexsym}
\usepackage{amsthm}
\usepackage{hyperref}
\usepackage{pstricks}
\newtheorem{assumpt}{Assumption}

\newtheorem{theorem}{Theorem}[section]
\newtheorem{proposition}{Proposition}[section]
\newtheorem{lemma}{Lemma}[section]

\title{Coexistence of stable limit cycles in a generalized Curie-Weiss model with dissipation}

\author{Luisa Andreis\thanks{WIAS-Weierstrass institute, Mohrenstr. 39, 10117 Berlin (Germany); e-mail address: luisa.andreis@wias-berlin.de} \and Daniele Tovazzi\thanks{Dipartimento di Matematica ``T. Levi-Civita'', Universit\`a degli Studi di Padova, Via Trieste 63, 35121 Padova (Italy); e-mail address: tovazzi@math.unipd.it}
} 

\begin{document}

\maketitle

\begin{abstract}
\noindent In this paper, we modify the Langevin dynamics associated to the generalized Curie-Weiss model by introducing noisy and dissipative evolution in the interaction potential. We show that, when a zero-mean Gaussian is taken as single-site distribution, the dynamics in the thermodynamic limit can be described by a finite set of ODEs. Depending on the form of the interaction function, the system can have several phase transitions at different critical temperatures. Because of the dissipation effect, not only the magnetization of the systems displays a self-sustained periodic behavior at sufficiently low temperature, but, in certain regimes, any (finite) number of stable limit cycles can exist. We explore some of these peculiarities with  explicit examples.

\vspace{0.3cm}

\end{abstract}

%----------------------------------------------------------------------------------------
%	ARTICLE CONTENTS
%----------------------------------------------------------------------------------------
	
%\linenumbers

%================================================================================
\section{Introduction}
%================================================================================

Occurrence of rhythmic behaviors is one of the most interesting phenomena observed in complex systems emerging from life science. Since it is natural to model such systems by means of large families of interacting units, one question is how periodic behaviors emerge { at macroscopic level} when the single units have no tendency to behave periodically. In this framework, we say that we observe a \emph{self-sustained periodic behavior} if there is a phase in which the equation for the evolution of the macroscopic law has a stable periodic solution, without the action of any periodic force: { examples  of this phenomenon come from biology, ecology and socioeconomics \cite{Ch08,Er10,Tu92,We12}}. This periodicity, even if often detectable by numerical simulations, it is a peculiarity of the thermodynamic limit and it is quite hard to formally prove it, due to the infinite dimensional nature of the problem.

Some recent works~\cite{DaGiRe14, GiPo15} investigate the minimal hypothesis needed to create self-sustained periodic behavior in mean field interacting particle systems. One key step is to consider interactions that favor cooperation among units of the same type, but the reversibility of the dynamics seems to be in contrast with the occurrence of periodic behavior~\cite{BeGiPa10, GiPo15}. {Therefore, a number of different mechanisms that perturb classical reversible models have been proposed}. %In~\cite{} authors consider  one among the others the stochastic Kuramoto model. They propose different dynamics and interactions for which they prove periodicity with a very general method. In particular 
{For example,} the role of \emph{noise} appears to be crucial: it may induce periodicity in systems whose deterministic version do not admit periodic solutions, i.e. \emph{noise-induced} periodicity~\cite{DaGiRe14, Sch85noise,  ToHeFa12} or it enhances periodicity in dynamical systems already proved to have limit cycles, i.e. \emph{excitability by noise}~\cite{CoDaFo15, LiGaNeSc04}. Moreover, the add of {disorder} in the initial phase of each rotator for the Kuramoto and other active rotators models, see~\cite{AcBoViRiSp05, CoDa12, GiPo15}, gives origin to oscillating behavior of the stationary solutions. In multi-populated models, it is instead a suitable structure of the {interaction network}, with or without the add of a {delay}, that produces periodicity, see~\cite{CoFoTo16, DiLo16, Tou14}.

In this paper, {we focus on the mechanism that induces self-sustained periodic behavior by introducing a dissipation term in the interaction energy. In~\cite{DaFiRe13}}, the authors consider a dynamical Curie-Weiss model where each particle has its own \emph{local field} that undergoes a \emph{diffusive and dissipative} dynamics. %The particle system that they study is the  with a dissipative term %studied in~\cite{DaFiRe13, }, which already 
 In this case, due to the mean field interaction and of the spin-flip dynamics, it is possible to completely analyze the limiting system: the authors show that the introduction of dissipation prevents the system from relaxing to a magnetized equilibrium and this gives rise to a periodic orbit via a \emph{Hopf bifurcation}. 
The general approach to introduce dissipation in systems of interacting diffusions is presented in~\cite{CoDaFo15}, together with the study of limit cycles in the dissipative counterpart of a model of cooperative behavior introduced by Dawson~\cite{Da83}. Starting from these works, %it seems natural to test the role of dissipation in creating limit cycles on a more general class of systems. Indeed, 
a natural question is whether the add of a dissipative term { in a cooperative interaction} is a robust tool that is able to create stable periodic orbit in systems otherwise in equilibrium. In this view, we aim to define models showing additional features than the ones above, but still analytically  tractable. We consider the class of \emph{generalized Curie-Weiss models}, defined { in~\cite{ElNe78} and further studied in~\cite{EiEl88}}, being a straightforward extension of the classical Curie-Weiss Gibbs measure. It replaces the quadratic interaction function by a more general one and the spin-$\frac{1}{2}$ single site distribution by a symmetric distribution on $\mathbb{R}$. We define a dissipated version of the Langevin dynamics for this model and we analytically study the case where the single site distribution of spins is given by a Gaussian distribution. This case is of particular interest because the limiting behavior is given by a simple Gaussian process, reducing the nonlinear problem to a three-dimensional one. Moreover, the dissipation originates a Li\'enard system, for which an extended literature on limit cycles is present. Thanks to these features, we confirm that the dissipative mechanism is able to create self-sustained periodic behavior. In particular, we prove that coexistence of more than one stable limit cycles is possible in systems of this type, that is an interesting extension of existing models with dissipation in which {at most one attractive periodic solution is possible}. The model considered in this paper seems to be {more general and flexible}, since we give sufficient conditions on the interaction function in order to have as many periodic orbit as wanted.

The paper is organized as follows: in Section~\ref{the_model} we define the model with general single site distribution at a microscopic and at a macroscopic level and we state the main results on the dynamics with the Gaussian single site distribution. We analyze the situation in which more than one stable periodic orbit coexist in the limiting dynamics and we give sufficient conditions on the interaction function for this to happen. We illustrate the many different regimes that may be induced by the appropriate tuning of the interaction function with the qualitative analysis of an explicit example that shows coexistence of stable limit cycles. In Section~\ref{Sec:proof_main_theorem} we prove the main result on the existence of periodic solution to the macroscopic equation in certain regimes, through the study of a planar dynamical system. In Section~\ref{SEC:prop_chaos} we
prove well-posedness of the macroscopic nonlinear equation and the link between microscopic and macroscopic equation via propagation of chaos.
%its mean field limit, in Section~\ref{gaussian_distr} we focus on the Gaussian single site distribution and we study the dynamics and the equilibria with and without the dissipation term, in Section~\ref{limit_cycles} .

\section{The model and main results}\label{the_model}
Let us consider a generalized Curie-Weiss model, see \cite{EiEl88}, that is a sequence of probability measures on $\mathbb{R}^N$, for $N=1,2,\dots$, given by
\begin{equation}\label{gibbs_generalized}
\mathbf{P}_{N,\beta}(dx_1,\dots,dx_N)=\frac{1}{Z_N(\beta)}\exp\left(N\beta g\left(\sum_{i=1}^N\frac{x_i}{N}\right)\right)\prod_{i=1}^N\rho(x_i)dx_i,
\end{equation}
where $\rho$ is the density function of a symmetric probability measure on $\mathbb{R}$ {with full support} representing the single-site distribution of a spin in absence of interaction, $g$ is the interaction function, $\beta$ is the inverse  temperature of the model and $Z_N(\beta)$ is the normalizing constant. For each $N$ fixed, a Langevin dynamics associated to \eqref{gibbs_generalized} is a diffusion process $X^N$ with values in $\mathbb{R}^N$ such that $\mathbf{P}_{N,\beta}$ is its unique invariant measure, i.e. $X^N$ is solution to the following systems of \mbox{SDE}
\begin{equation}\label{X^N_no_dissip}
dX^N_i(t)=\frac{\beta}{2} g'\left(\frac{\sum_{j=1}^N X^N_j(t)}{N}\right)dt{+}\frac{\rho'(X^N_i(t))}{2\rho(X^N_i(t))}dt+dB^i_t,
\end{equation}
where $\{B^i\}_{i=1,\dots,N}$ is a family of independent $1$-dimensional Brownian motions. The dynamics in \eqref{X^N_no_dissip} represent an interacting particle system where each particle follows its own dynamics, given by the last two terms on the right-hand side, and it experiences an interaction which depends on the empirical mean of the system $m^N(t)\colon =\frac{\sum_{j=1}^N X^N_j(t)}{N}$.
 By following the approach in~\cite{DaFiRe13, CoDaFo15}, we aim to define a dynamical model of interacting diffusions where the interaction %that affects a single particle
 depends on a ``perceived magnetization'' instead of on $m^N(t)$. Indeed, we introduce the variables $\lambda^N_i$, for $i=1,\dots,N$, that evolve as the magnetization of the system but they undergo a dissipative and diffusive evolution
 and we substitute them in~\eqref{X^N_no_dissip}, obtaining
 \begin{equation}\label{lambda}
\displaystyle{ \left\{\begin{array}{l}
dX^N_i(t)=\frac{\beta}{2} g'\left(\lambda^N_i(t)\right)dt{+}\frac{\rho'(X^N_i(t))}{2\rho(X^N_i(t))}dt+dB^{i,1}_t,\\
\\
d\lambda^N_i(t)=-\alpha\lambda^N_i(t)dt+DdB^{i,2}_t+dm^N(t),
\end{array}\right.}
\end{equation}
 where $\{(B^{i,1},B^{i,2})\}_{i=1,\dots,N}$ are independent Brownian motions. When $\alpha=D=0$ the ``perceived magnetization'' is equivalent to the total magnetization of the system and the evolution~\eqref{lambda} is equivalent to~\eqref{X^N_no_dissip}. 
Clearly for the well-posedness of~\eqref{lambda}, the quantities involved must satisfy certain assumptions, that we enumerate here.
\begin{assumpt}\label{assumption_gen} Let $\alpha, D\geq0$ and $\beta>0$ and the functions $g$ and $\rho$ satisfy the following:
%The constants $\alpha$ and $D$ are nonnegative and the inverse absolute temperature $\beta$ is positive. Furthermore, the function $g$ and the probability measure $\rho$ satisfy the following conditions.
\begin{itemize}
\item[i)] $g\colon \mathbb{R}\rightarrow\mathbb{R}_{\geq0}$ is an even, $C^2(\mathbb{R})$ function, analytic outside a countable set of points, strictly increasing on $[0,\infty)$ with $g(0)=0$. 
%It is \emph{two-sided real analytic}, i.e. $\forall$ $x$~$\in$~$\mathbb{R}$ there exists $\delta>0$ and two real analytic functions $g_1$ and $g_2$ on $(x-\delta,x+\delta)$ such that $$g\colon=\left\{\begin{array}{ll}g_1&\text{ on }(x-\delta,x]\\g_2&\text{ on }[x,x+\delta).\end{array}\right.$$
\item[ii)] $g'$ is uniformly Lipschitz continuous, i.e. there exists a finite constant $L\geq0$ such that for all $x,y$ $\in$ $\mathbb{R}$
$$
|g'(x)-g'(y)|\leq L |x-y|.$$

\item[iii)] $\rho$ is the density function of a symmetric Borel probability measure on $\mathbb{R}$ {with full support}, absolutely continuous w.r.t. the Lebesgue measure, such that
%and, by abuse of notation, we denote its density function with $\rho(x)$. We require 
$\log(\rho(x))$ $\in$ $C^2$ and there exists $K>0$ s.t. for all $x,y$ $\in$ $\mathbb{R}$
$$
(x-y)\left(\frac{\rho'(x)}{\rho(x)}-\frac{\rho'(y)}{\rho(y)}\right)\leq K(1+(x-y)^2).$$
%We require that the function $\rho$ is $C^1(\mathbb{R})$ and that there exists a finite $K>0$ such that the function $\log (\rho(\cdot))$ is concave for all $|x|>K$.
\item[iv)] There exists a symmetric, nonconstant, convex function $h$ on $\mathbb{R}$ such that
\begin{align}
g(x)\leq h(x)\, \text{ for all }x\in\mathbb{R},\nonumber\\
\int_{\mathbb{R}}e^{a h(x)}\rho(x)dx<\infty\, \text{ for all } a>0.\label{integrabilita_g}
\end{align}
\end{itemize}

\end{assumpt}
These are technical assumptions that mainly come from the definition of the generalized Curie-Weiss model in \cite{EiEl88} and from the classical theory of mean field models. Under Assumption~\ref{assumption_gen} the generalized Curie-Weiss model is well-defined and, since the interactions %in \eqref{inf_gen_N} 
are of mean field type, we expect a propagation of chaos result to hold.
%if $\alpha=D=0$, the Markov process with generator \eqref{inf_gen_N} describes a Langevin dynamic with stationary measure $\mathbf{P}_{N,\beta}\times \delta_{0}\times\dots\times \delta_{0}$ where $\mathbf{P}_{N,\beta}$ is the Gibbs measure \eqref{gibbs_generalized} on $\mathbb{R}^N$ and $\delta_{0}$ is the Dirac delta centered in the stationary magnetization $m^N=0$. 
Therefore, we define the nonlinear Markov process $(X,\lambda)$ on $\mathbb{R}^2$, that is solution of the following nonlinear \mbox{SDE}:
\begin{equation}\label{nonlinear_limit}
\left\{\begin{array}{l}
dX_t=\frac{\beta}{2} g'(\lambda_t)dt{+}\frac{\rho'(X_t)}{2\rho(X_t)}dt+dB^1_t\\
d\lambda_t=-\alpha \lambda_t dt+\langle\mu_t(x,l),\frac{\beta}{2} g'(l){+}\frac{\rho'(x)}{2\rho(x)} \rangle dt+DdB^2_t\\
\mu_t=Law(X_t,\lambda_t),
\end{array}
\right.
\end{equation}
where $B=(B^1,B^2)$ is a two dimensional Brownian motion. Well-posedness of \eqref{nonlinear_limit} and the property of propagation of chaos are proved, with a rather standard approach, in Section~\ref{SEC:prop_chaos}. Collective periodic behavior is a phenomenon which appears only in this mean field limit. Indeed, the finite dimensional dynamics of the $N$ particles system forbid the existence of more than one stationary solution, while the nonlinear process may admit invariant sets of measures, which identify as stable periodic solution. For this reason, in the following we study the long-time behavior of solutions of~\eqref{nonlinear_limit} under certain additional assumptions.
% Their proofs follow a rather standard approach in the theory of McKean-Vlasov equations and they are reported in the Appendix.

\subsection{Existence of limit cycles in the Gaussian dynamics}\label{gaussian_distr}
We choose to focus on the {\em Gaussian dynamics}, i.e. we choose as single site distribution of spins the Normal distribution with mean zero and variance $\sigma^2$. This means that we add the following set of assumptions.
\begin{assumpt}\label{assumption_gauss} Let $D=0$, $\sigma>0$ and 
$$\rho(x)\colon=\frac{1}{\sqrt{2\pi \sigma^2}}e^{-\frac{x^2}{2\sigma^2}}.$$
\begin{itemize}
\item[i)] There exists a symmetric, nonconstant, convex function $h$ on $\mathbb{R}$ such that
\begin{align}
g(x)\leq h(x)\, \text{ for all }x\in\mathbb{R},\nonumber\\
\int_{\mathbb{R}}e^{a h(x)}e^{-x^2}dx<\infty\, \text{ for all } a>0.\label{integrabilita_g}
\end{align}
\item[ii)] Let the initial condition measure be of the form 
$$
\mu_0(dx,d\lambda)=\nu_0(dx)\times \delta_{\lambda_0}(d\lambda),
$$
where $\nu_0$ is a square-integrable measure  on $\mathbb{R}$ and $\delta_{\lambda_0}$ is a Dirac delta centered in  $\lambda_0$~$\in$~$\mathbb{R}$.
\end{itemize}
\end{assumpt}
%This drastically simplifies the treatment, indeed, 
Under Assumptions~\ref{assumption_gen} and~\ref{assumption_gauss}, the nonlinear process $(X(t),\lambda(t))_{t\geq0}$,  is solution of the following nonlinear \mbox{SDE}:
%the solution of \eqref{nonlinear_limit}, is a Gaussian process. 
\begin{equation}\label{nonlinear_limit_gaussian}
\left\{\begin{array}{l}
dX_t=\frac{\beta}{2} g'(\lambda_t)dt-\frac{X_t}{2\sigma^2}dt+dB_t,\\
\frac{d\lambda_t}{dt}=-\alpha \lambda_t +\frac{\beta}{2} g'(\lambda_t)-\frac{m_t}{2\sigma^2}, \\
\mu_t=Law(X_t,\lambda_t)\,\, \text{ and }\,m_t=\langle \mu_t(dx,dl),x\rangle,
\end{array}
\right.
\end{equation}
for $\{B_t\}$ Brownian motion. Clearly $\mu_t$ preserves the product form $\nu_t(dx)\times \delta_{\lambda_t}(d\lambda)$ and the resulting process is a Gaussian process, specifically it is completely described by the initial condition $\mu_0$ and the quantities $\{(m_t,v_t,\lambda_t)\}_{t\geq0}$, where  $v_t=\mathbf{Var}[X_t]$.
In the following we study the stability properties and the long-time behavior of~\eqref{nonlinear_limit_gaussian} in this regime, where the $\lambda$ component is completely deterministic.

\begin{theorem} \label{THM:convergencegaussianmeasure2} Fix $\sigma^2>0$ and $\alpha>0$, for every $\beta>0$
%and let $\Lambda(\beta)$ as in \eqref{zero_points}, 
the process $(X_t,\lambda_t)$ described by \eqref{nonlinear_limit_gaussian} has exactly one stationary solution of the form $\nu^*(dx)\times \delta_{\lambda^*}(d\lambda)$ and it is given by the measure
$$\mu^*_{(0,0)}(dx,dl)\colon=\frac{1}{\sqrt{2\pi \sigma^2}}e^{-\frac{x^2}{2\sigma^2}}dx\times\delta_0(dl).$$
\begin{itemize}
\item[i)] There exists $\beta^*>0$ such that $\forall$~$\beta$~$\in$~$(0,\beta^*)$ {$\mu^*_{(0,0)}$ is the unique stationary solution for~\eqref{nonlinear_limit_gaussian}}. Moreover, for all $\mu_0(dx,d\lambda)=\nu_0(dx)\times \delta_{\lambda_0}(d\lambda),
$ square-integrable initial conditions with $\lambda_0 \in\mathbb{R}$,
\begin{equation}\label{long-time_origin_stable}
\lim_{t\rightarrow \infty}\|\mu_t(\cdot)-\mu^*_{(0,0)}(\cdot)\|_{TV}=0.
\end{equation}

\item[ii)] If $g''(0)>0$, let $\beta_H\colon =\frac{2\alpha+\frac{1}{\sigma^2}}{g''(0)}$, then for $\beta>\beta_H$ there exists at least one closed curve $\gamma$~$\in$~$\mathbb{R}^2$ such that the set
$$
\Gamma(\gamma)=\left\{\mu^*_{(m,\lambda)}(dx,dl)\colon=\frac{1}{\sqrt{2\pi \sigma^2}}e^{-\frac{(x-m)^2}{2\sigma^2}}dx\times\delta_{\lambda}(dl),\,\, \text{ for all }\, (m,\lambda)\in\gamma\right\}
$$
is an invariant set for the dynamics \eqref{nonlinear_limit_gaussian}. Moreover, it exists an open set $U(\gamma)\in\mathbb{R}^2$ such that for all $\mu_0(dx,d\lambda)=\nu_0(dx)\times \delta_{\lambda_0}(d\lambda),
$ square-integrable initial conditions with $(\langle \mu_0,x\rangle, \lambda_0)$~$\in$~$U(\gamma)$, it holds
\begin{equation}\label{long-time2}
\lim_{t\rightarrow \infty}\inf_{(m,\lambda)\in\gamma}\|\mu_t(\cdot)-\mu^*_{(m,\lambda)}(\cdot)\|_{TV}=0.
\end{equation}
\item[iii)] If $g''(0)>0$ and there exists $C>0$ such that for all $x\in(C,\infty)$ the function $g'(x)$ is concave, then there exists a $\beta_{UC}$ such that for all $\beta>\beta_{UC}$ there exists a unique closed curve $\gamma$~$\in$~$\mathbb{R}^2$ such that $\Gamma(\gamma)$ is invariant for the dynamics~\eqref{nonlinear_limit_gaussian} and \eqref{long-time2} holds with $U(\gamma)=\mathbb{R}^2\setminus \{(0,0)\}$.
\end{itemize}
 
%Moreover, for all $\mu_0(dx,d\lambda)=\nu_0(dx)\times \delta_{\lambda_0}(d\lambda),$ square-integrable initial conditions with $\lambda_0 \in\mathbb{R}$, \begin{equation}\label{long-time2}\lim_{t\rightarrow \infty}\inf_{(m,\lambda)\in\gamma}\|\mu_t(\cdot)-\mu^*_{(m,\lambda)}(\cdot)\|_{TV}=0,\end{equation} where $\gamma$ is the attractor of the trajectory starting from $(\langle \mu_0,x\rangle, \lambda_0)$ in the dynamical system \eqref{ODE_diss}; here with $\gamma$ we mean either a limit cycle or simply the origin.
\end{theorem}

{In Theorem~\ref{THM:convergencegaussianmeasure2} we describe different parameter regimes of the dynamics~\eqref{nonlinear_limit_gaussian}. It is clearly not an exhaustive picture, but it already shows some interesting features of the process. Point \emph{i)} says that, at sufficiently high temperatures, there is a unique equilibrium for the process, which is the disordered situation. In Point \emph{ii)} we highlight the existence, for certain types of interaction function, of what can be identified as a \emph{Hopf bifurcation}. Then, for sufficiently low temperatures we see the existence of a set of measures which is invariant under the dynamics. This corresponds to a periodic solution of the system~\eqref{nonlinear_limit_gaussian} and it is stable in the sense that it has a certain domain of attraction. Point \emph{iii)} proves that, for a certain class of interaction function, there exists a regime in which we have uniqueness of the periodic solution and this solution attracts all the trajectories starting with a deterministic $\lambda$, except for the stationary one $\mu^*_{(0,0)}$. In the following subsection we underline the existence of many other regimes and in particular we give sufficient conditions for the emergence of regimes with an arbitrary number of periodic solution. }

\subsection{Coexistence of stable limit cycles}\label{limit_cycles}
As pointed out in \cite{DaFiRe13}, collective periodic behavior can be observed in certain regimes when dissipation is introduced in a classical Curie-Weiss model. In that particular model, a Hopf bifurcation in a dynamical system identifies the transition from disorder to a phase in which a unique globally stable limit cycle is present. The model introduced here generalizes this scheme: the proof of Theorem~\ref{THM:convergencegaussianmeasure2} will show that the study of limit cycles depends (as in~\cite{DaFiRe13}) on the study of a planar dynamical system of Li\'{e}nard type, see~\cite{Od96, ChLlZh07}. In point \emph{ii)} of Theorem~\ref{THM:convergencegaussianmeasure2} we describe what happens when a Hopf bifurcation occurs in this system but, according to the form of the interaction function $g(x)$, the limiting dynamics may display a richer behavior. The most interesting phenomena %which can be observed 
are the appearance of limit cycles when the origin is still locally stable and the coexistence of more than one limit cycles. These orbits, in general, are created and destroyed through global bifurcations. Let us give a general idea of what may happen, depending on the interaction function $g$. As we prove in Section~\ref{Sec:proof_main_theorem}, the limit cycles of~\eqref{nonlinear_limit_gaussian} depend on the following dynamical system of Li\'enard type:
\begin{equation}\label{ODE_lienard_change_variable}
\left\{
\begin{array}{l}
\dot{y_t}=-\frac{\alpha}{2\sigma^2}\lambda_t,\\
\dot{\lambda_t}=y_t-f_{\alpha,\beta}(\lambda_t), \\
\end{array}
\right.
\end{equation}
with 
\begin{equation}\label{funct_f}
f_{\alpha,\beta}(x)\colon=\left(\alpha+\frac{1}{2\sigma^2}\right) x -\frac{\beta}{2} g'(x).
\end{equation}
 The number of limit cycles in a Li\'{e}nard system such as \eqref{ODE_lienard_change_variable} mainly depends on the function $f_{\alpha,\beta}(x)$: some tools to determine the exact number of limit cycles are available in literature (see \cite{Od96}, \cite{ChLlZh07} and references therein) but their application may be cumbersome in a general setting, since several features of the function $f_{\alpha,\beta}(x)$ should be studied, such as the position of its zeroes, its local minima and maxima, their height and so on. However, playing with the form of the interaction function $g$, we can always create a Gaussian Curie-Weiss model with dissipation with a  customized number of phase transitions and of coexisting limit cycles. 
%\subsection{Number of limit cycles}
%Let us briefly underline the role of the function $g$ in the occurrence of limit cycles in the dynamics of \eqref{ODE_diss}. To this aim, we rewrite the Li\'enard system \eqref{ODE_lienard}:In the rich literature on Li\'enard system, we see that the form of the function $f_{\alpha,\beta}$ plays a fundamental role in the number of limit cycles of the system. 
In particular, from the results in~\cite{Od96}, we can state the following.
\begin{proposition}\label{prop:N_cicli}
Fix $\sigma^2,\alpha>0$ and suppose that there exists a $\beta^*$ such that the following conditions are satisfied:
\begin{itemize}
\item[i)] the function $f_{\alpha,\beta^*}$ has $K$ positive zeros $x_0\colon =0<x_1<\dots<x_K(<x_{K+1} \text{ a bound })$ at which it changes sign;
\item[ii)] for every $k=1,\dots,N$ there is a $C^1$ mapping $\phi_k\colon[x_{k-1},x_k]\rightarrow[x_{k},x_{k+1}]$ such that
$$
\phi_k(x)\phi_k'(x)\geq x\,\text{ and }\, |f_{\alpha,\beta^*}(\phi_k(x))|\geq |f_{\alpha,\beta^*}(x)|;$$
\item[iii)] the function $f_{\alpha,\beta^*}$ on each interval $[x_{k-1},\phi_{k-1}(x_{k-1})]$ for $2\leq k\leq K+1$ has an extremum at a unique point $y_k$ and its derivative is weakly monotone.
\end{itemize} 
Then the generalized Curie-Weiss model with dissipation has at least one regime in which it has exactly $N$ invariant sets of the form $\Gamma(\gamma_i)$ for some closed curves $\{\gamma_i\}_{i=1,\dots,K}$. These are concentric cycles such that the outer is attractive, then the others alternate as repulsive and  attractive, respectively. 
\end{proposition}
The proof of this result is a simple application of the results in~\cite{Od96}. It is easy to see that the function $f_{\alpha,\beta}$
depends on the choice of the interaction function $g$. Since Assumptions~\ref{assumption_gen} and \ref{assumption_gauss} are not very restrictive, $g$  can be manipulated to obtain a system that admits a regime with the desired number of limit cycles.

\subsection{An explicit example }
With an explicit example, we show how we can manipulate the interaction function $g$ in order to observe the coexistence of two stable limit cycles. Let us define the function 
\begin{equation}\label{g_example}
g(x)=\tanh\left(ax^2+bx^4+cx^6\right),
\end{equation}
with $a,b,c$ suitable constants such that $g$ stays strictly increasing on $[0,\infty)$. Fix $\sigma^2>0$, then the pair $(g,\rho)$, with $\rho\sim\mathcal{N}(0,\sigma^2)$ clearly satisfies Assumptions~\ref{assumption_gen} and~\ref{assumption_gauss} and it defines a generalized Curie-Weiss model.  We consider two triplets of constants $(1/2,-1,2)$ and $(1,1,0)$  in order to observe some particular regimes that do not exist for the classical Curie-Weiss model with dissipation. {Let us underline that, in the following, the qualitative study of the behavior of the process is made by means of the study of a system of ODE. We postpone to Section~\ref{Sec:proof_main_theorem} the proof that this is indeed sufficient.}
\subsubsection{Case A: triplet $\left(\frac{1}{2},-1,2\right)$} We see from Figure~\ref{fig:g} the changes in the concavity of $g'(x)$. {This causes an interesting behavior also in the dynamics without dissipation. Indeed, if we consider $\alpha=0$ and $\lambda_0=\langle \nu_0,x\rangle$, the dynamics in~\eqref{nonlinear_limit_gaussian} is simply the mean field limit of the Langevin dynamics described in~\eqref{X^N_no_dissip}. This is a Gaussian nonlinear process and it is clear that it suffices to study the following system
\begin{equation}\label{ODE_NOdiss_with_v}
\left\{
\begin{array}{l}
\dot{m_t}=\frac{\beta}{2}g'(m_t)-\frac{m_t}{2\sigma_2},\\
\dot{v_t}=1-\frac{v_t}{\sigma^2},\\
\end{array}
\right.
\end{equation}
to analyze its stationary measures and its long-time behavior. We analyze in details only the equilibria of the first row in~\eqref{ODE_NOdiss_with_v}, since the two variables are independent. The ODE is characterized by three critical values of $\beta$ and the four following regimes:}
\begin{itemize}
\item[-] for $\beta<\beta_1$ the origin is a global attractor;
\item[-] for $\beta$~$\in$~$(\beta_1,\beta_2)$ the origin is locally stable, but there are four other equilibrium points $-x_2<-x_1<0<x_1<x_2$, such that $\pm x_2$ are stable and $\pm x_1$ are unstable;
\item[-] for $\beta$~$\in$~$(\beta_2,\beta_3)$ the origin becomes unstable and two additional  stable equilibrium points appear, $\dots -x_1 <-x_3<0<x_3<x_1 \dots$;
\item[-] for $\beta=\beta_3$ the pairs of equilibrium points $\{x_3,x_1\}$ and $\{-x_3,-x_1\}$ collapse and disappear, such that for $\beta>\beta_3$ there are three equilibrium points $-x_2<0<x_2$, the outer two are stable and the origin is unstable.
\end{itemize}
\begin{figure}[h!] 
\centering
   \includegraphics[width=0.49\linewidth]{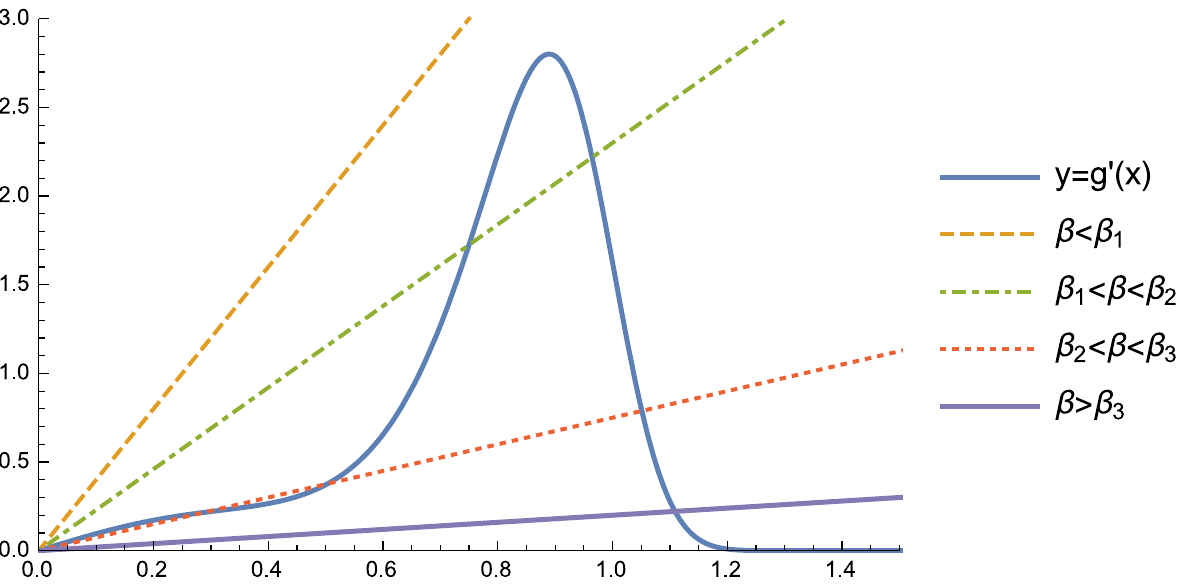}
      % \caption{(a)}
   \includegraphics[width=0.49\linewidth]{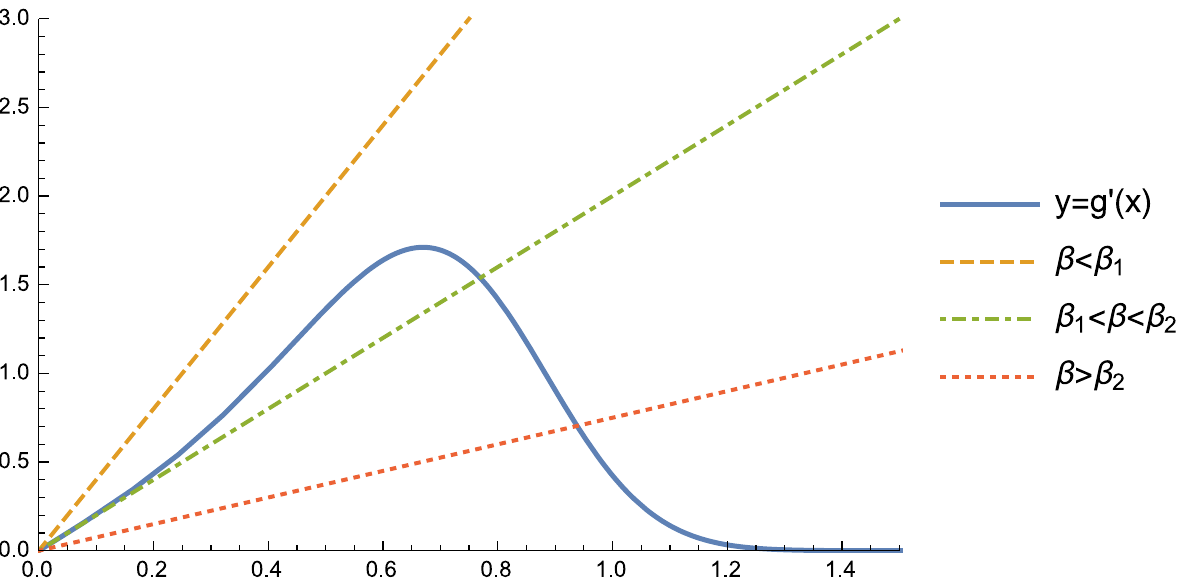}
   \caption{The plot of the function $g'$ and of  lines $y=\frac{1}{\beta\sigma^2}x$ for different values of $\beta$. The number of intersections gives the number of equilibrium points in the positive axes. Left: the case~A. Right: the case~B.  }\label{fig:g}
\end{figure}

The exact critical values may be obtained numerically, and the behavior of the dynamical system is clear from~\eqref{ODE_NOdiss_with_v}.\\ In this case, the dissipated dynamics {(identified from the study of system~\eqref{ODE_diss})} actually shows four different regimes as well, but the critical values $\hat\beta_1(\alpha)$, $\hat\beta_2(\alpha)$, $\hat\beta_3(\alpha)$ are not straightforwardly obtained with the same procedure of the elements of $CV$. To be precise, if $\beta_1$ corresponds to the smallest value of $\beta$ in which the line $y=\frac{x}{\sigma^2\beta}$ is tangent to the graph of $y=g'(x)$, the value $\hat\beta_1(\alpha)$ is strictly greater than the smallest value of $\beta$ such that the line  $y=\frac{2\alpha+\frac{1}{\sigma^2}}{\beta}x$ is tangent to the graph of $y=g'(x)$. This means that there exists a $\beta^*$ such that the line $y=\frac{2\alpha+\frac{1}{\sigma^2}}{\beta^*}x$ intersects the graph of $y=g'(x)$ but any limit cycle occurs. Nevertheless, the system displays a regime of \emph{coexistence of stable limit cycles}. Let us better explain the four regimes that we observe in system \eqref{ODE_diss} (actually the computations and the plots refer to system \eqref{ODE_lienard_change_variable}, since the link with the function $f_{\alpha,\beta}$ is more clear in this case).

\begin{figure}[h!] 
\centering
   \includegraphics[width=0.45\linewidth]{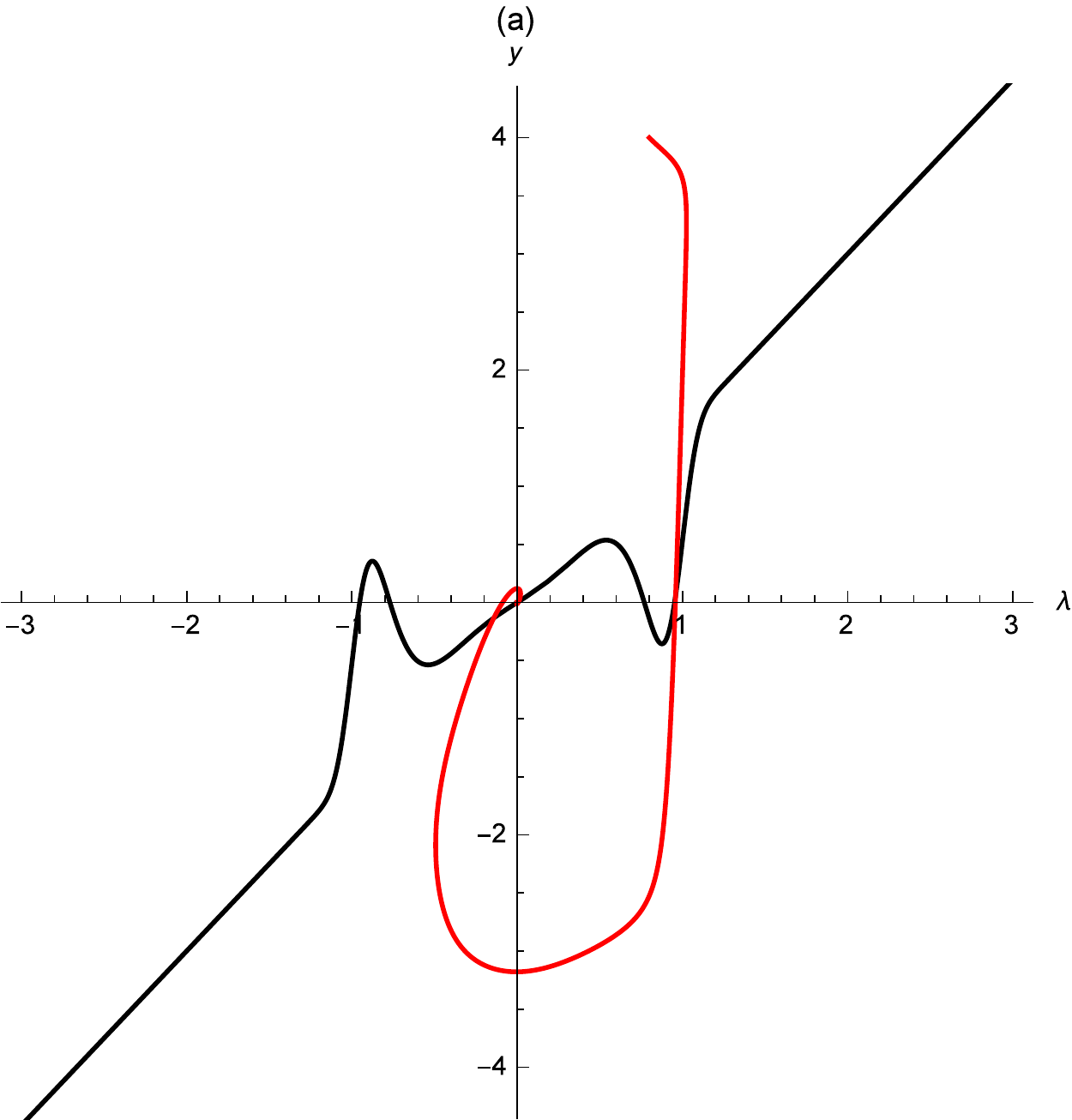}
   \includegraphics[width=0.45\linewidth]{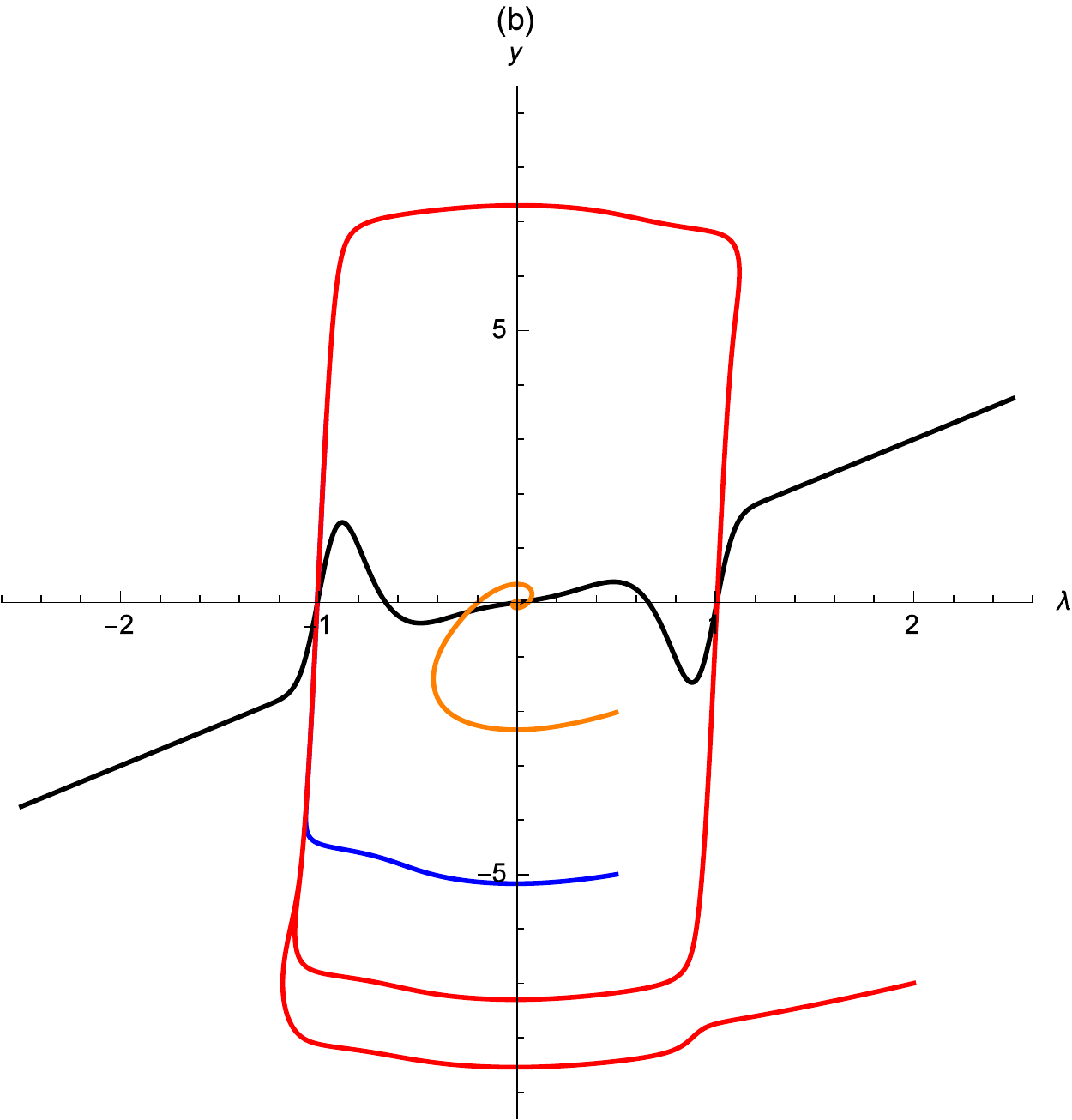}

  \includegraphics[width=0.45\linewidth]{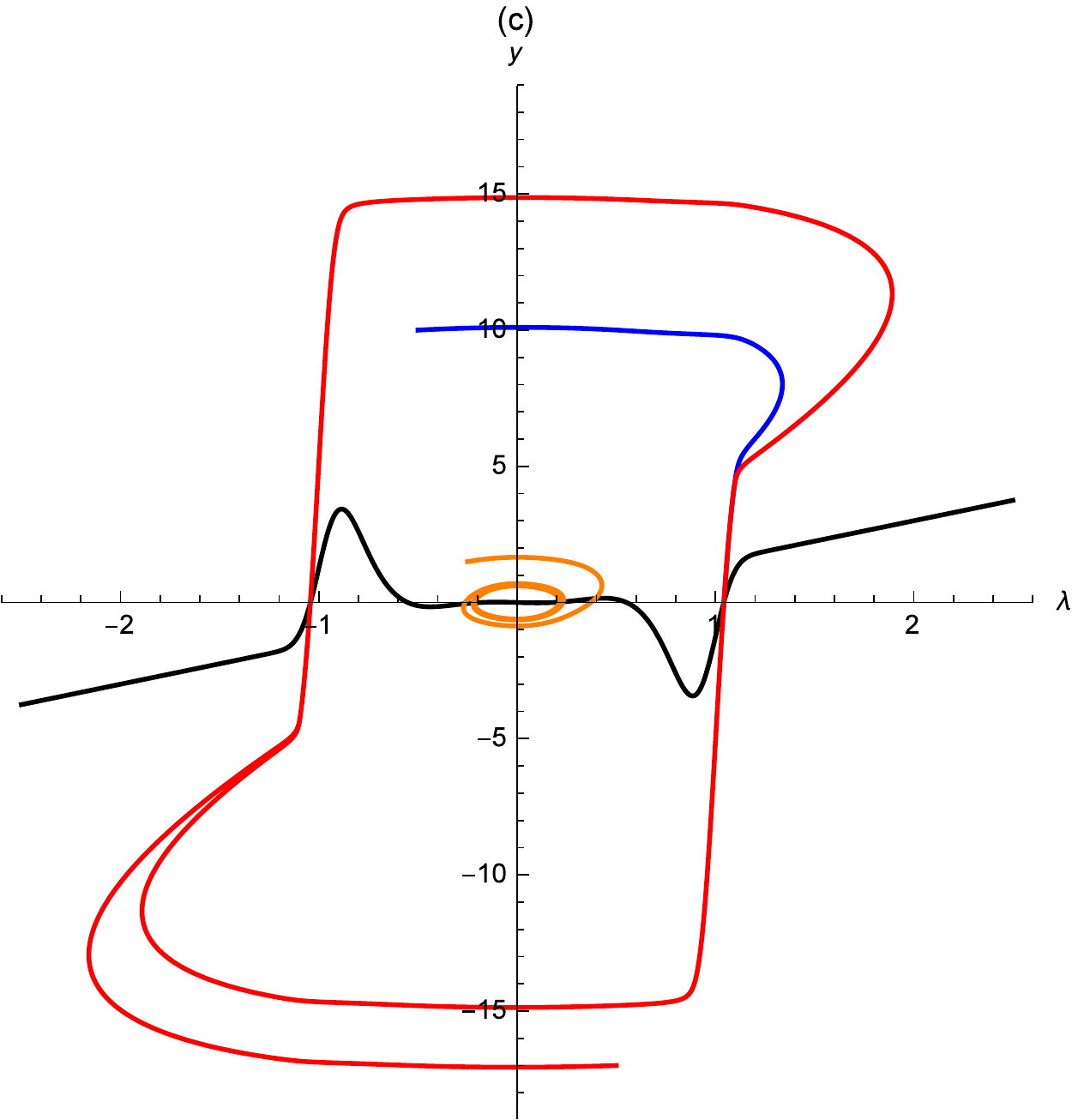}
  \includegraphics[width=0.45\linewidth]{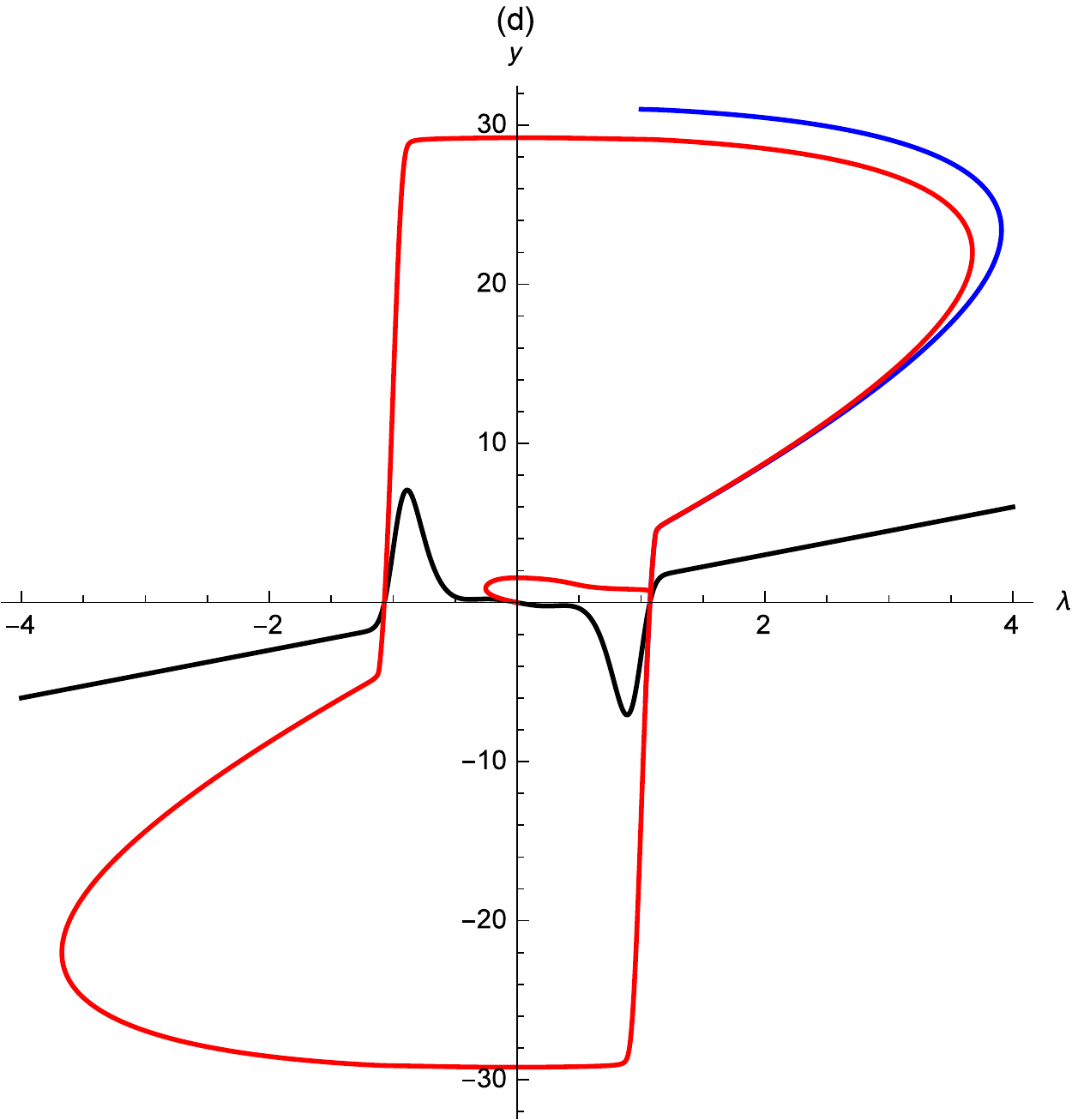}
      % \caption{(a)}
     \caption{\small Different regimes for case A. In all the pictures, the black line represents the graph of $y=f_{\alpha,\beta}(\lambda)$ and we fixed $\alpha=\sigma^2=1$. In (a), the regime $\beta<\hat{\beta}_1(\alpha)$ ($\beta=1.2$): the red line represents the solution starting from $\lambda(0)=1$, $y(0)=4$, which is definitely attracted by the globally stable origin. In (b), the regime $\beta \in (\hat{\beta}_1,\hat{\beta}_2)$ ($\beta =2$): the red-colored solution, starting from  $\lambda(0)=2$, $y(0)=-7$, and the blue-colored solution, starting from $\lambda(0)=0.5$, $y(0)=-5$, are attracted by a stable limit cycle. Here, the origin is locally stable (the orange-colored solution with initial condition $\lambda(0)=0.5$, $y(0)=-2$ is attracted by it) and its basin of attraction is surrounded by an unstable limit cycle. In (c), the regime $\beta \in (\hat{\beta}_2,\hat{\beta}_3)$ ($\beta =3.4$): the red and blue lines, here representing solution starting from $\lambda(0)=0.5$, $y(0)=-17$ and $\lambda(0)=-0.5$, $y(0)=10$ respectively, are again attracted by the outer cycle but now the origin is unstable and another stable cycle is born via the Hopf bifurcation. The orange-colored solution, with initial condition $\lambda(0)=-0.25$, $y(0)=1.5$, is attracted by the smallest cycle. The basins of attraction of the stable orbits is separated by an unstable cycle.  In (d), the regime $\beta >\hat{\beta}_3$ ($\beta =6$): only the external orbit is survived and it has become globally attractive, as shown by the red and blue lines, with initial conditions $\lambda(0)=0$, $y(0)=-0.005$ and $\lambda(0)=1.5$, $y(0)=31$ respectively. }\label{fig:casoA1}
\end{figure}

\begin{itemize}
\item[-] For $\beta<\hat\beta_1(\alpha)$ the origin is a global attractor. Notice that, numerically we can see that $\hat\beta_1(\alpha)$ is greater that the $\beta^*$ obtained in Theorem~\ref{thm:dyn_system_with_diss}; indeed it is not necessary that the function $f_{\alpha,\beta}(x)$ is strictly greater than zero for all $x>0$. It is reasonable to believe, see Figure \ref{fig:casoA1}(a), that for those $\beta$ such that the negative part of the function is ``small enough'' do not necessary give rise to a periodic orbit.
\item[-] For $\beta$~$\in$~$(\hat\beta_1(\alpha),\hat\beta_2(\alpha))$ the origin is locally stable and, through a global bifurcation, two periodic orbits have arised, the larger one is stable and the smaller one is unstable. Indeed, we can prove numerically that there exists a value $\beta^*$~$\in$~$(\hat\beta_1(\alpha),\hat\beta_2(\alpha))$ such that the function $f_{\alpha,\beta^*}$ satisfies the conditions of Theorem~A and~B in~\cite{Od96} for the existence of exactly two limit cycles. See Figure \ref{fig:casoA1}(b) in which we can observe the stable outer cycle and the attractivity of the origin.
\item[-] Notice that $\beta_H=\hat\beta_2(\alpha)$ from Theorem~\ref{thm:dyn_system_with_diss}. Therefore, for $\beta$~$\in$~$(\hat\beta_2(\alpha),\hat\beta_3(\alpha))$ we observe two attractive limit cycles, a smaller one spreading from the origin (that is now unstable) while the bigger one remains from the previous regime: see Figure \ref{fig:casoA1}(c). The basin of attraction of the two stable orbits are separated by a third unstable periodic orbit. This is the regime in which we see the coexistence of two stable periodic orbits, one inside the other. The existence of this regime is again a consequence of Theorem~A and~B in~\cite{Od96}, because we can numerically find a $\beta^{**}$  that satisfies the hypothesis for the existence of exactly three limit cycles, two stable and one unstable.
\item[-] For $\beta>\hat\beta_3(\alpha)$ we see that only the largest periodic orbit has survived. Indeed, for in $\hat\beta_3(\alpha)$ the smallest stable orbit and the unstable one collapse and disappear. Of course, we see that this $\hat\beta_3(\alpha)=\beta_{UC}$ defined in Theorem~\ref{thm:dyn_system_with_diss}; but from numerical evidence we suppose that for this value of $\beta$ the  function $f_{\alpha,\beta}$ has more than one single zero in the positive half-line, but that the other two zeros are not distant enough to admit the existence of the two inner orbits. Of course when $\beta$ is such that there exists a unique positive zero for $f_{\alpha,\beta}$, we analitically prove the existence and uniqueness of the limit cycle (see Theorem~\ref{thm:dyn_system_with_diss}) while for lower values we can only show it numerically, see Figure\ref{fig:casoA1}(d).
\end{itemize}

\subsubsection{Case B: triplet $\left(1,1,0\right)$} 
We see in Figure~\ref{fig:g} that the shape of $g'$ basically allows three different regimes for the case without dissipation. Indeed we have $\beta_1<\beta_2=\frac{2}{\sigma^2}$ and the three regimes are the following:
\begin{itemize}
\item[-] for $\beta<\beta_1$ the origin is a global attractor;
\item[-] for $\beta$~$\in$~$(\beta_1,\beta_2)$ there are five equilibrium points $-x_2<-x_1<0<x_1<x_2$, where $\pm x_1$ are unstable, while the others are stable;
\item[-] at $\beta=\beta_1$ the two points $\pm x_1$ collapse in the origin that becomes unstable, such that for $\beta>\beta_1$ the origin is unstable and the points $\pm x_2$ are stable.
\end{itemize}
We treat this example in the dissipated case \eqref{ODE_diss} (by means of the Li\'enard system \eqref{ODE_lienard_change_variable}). We expect three regimes and, in particular, we will observe an atypical behavior at the Hopf bifurcation, where we will not have a small limit cycle bifurcating from the origin, but the already existing stable limit cycle that becomes a global attractor. In Figure \ref{fig:casoB} we compare the regimes immediately below and above the Hopf bifurcation.

\begin{figure}[h!] 
\centering
   \includegraphics[width=0.45\linewidth]{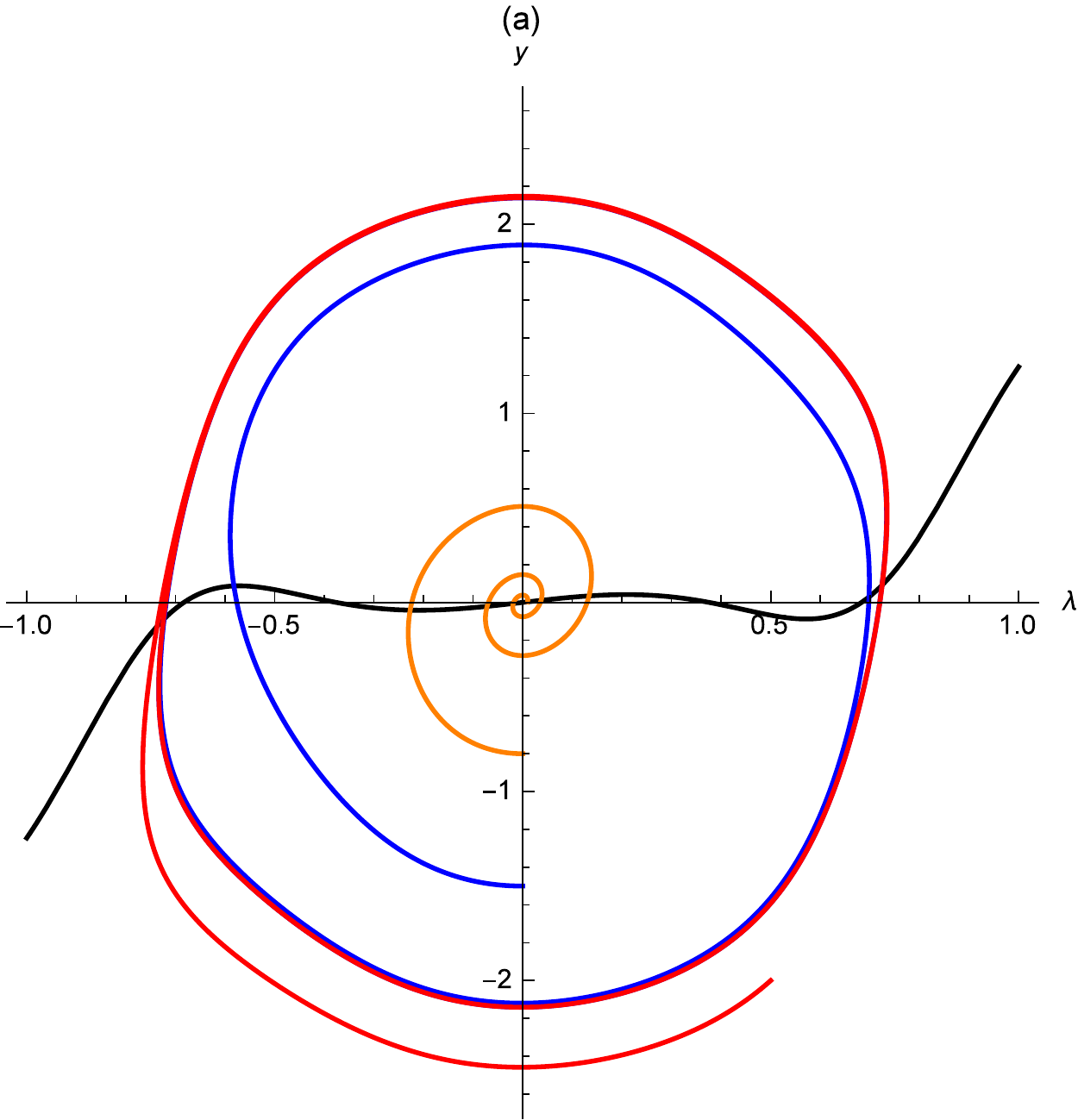}
   \includegraphics[width=0.45\linewidth]{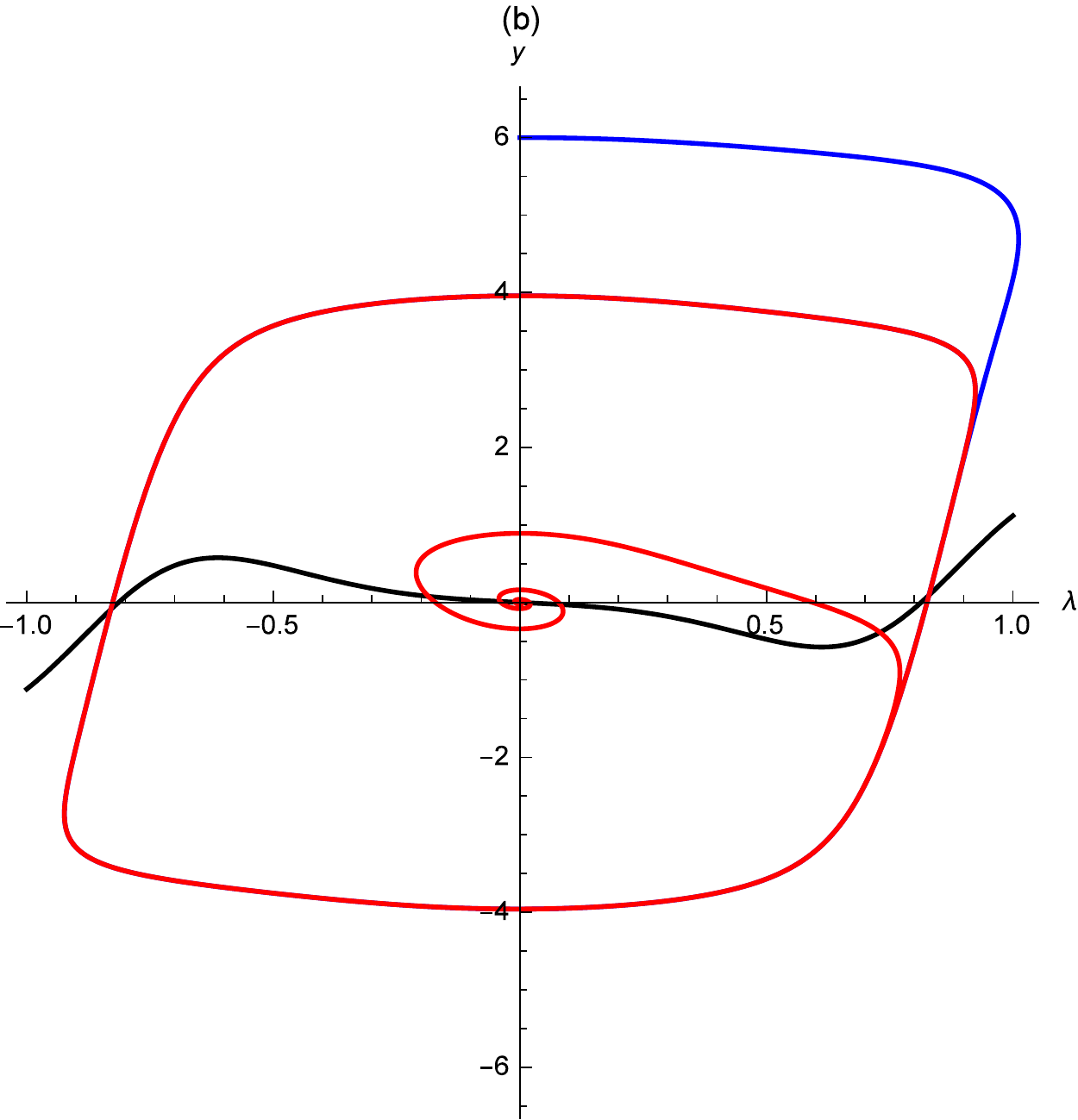}
   \caption{\small Different regimes for case B close to the Hopf bifurcation. In both pictures, the black line represents the graph of $y=f_{\alpha,\beta}(\lambda)$ and we fixed $\alpha=\sigma^2=1$. In (a), the regime $\beta\in(\hat{\beta}_1(\alpha),\beta_H)$ ($\beta=1.2$): the situation is qualitatively the same of Figure \ref{fig:casoA1}(b). The red, blue and orange lines represent solution starting from $\lambda(0)=0.5$, $y(0)=-2$, $\lambda(0)=0$, $y(0)=-1.5$ and $\lambda(0)=0$, $y(0)=-0.8$ respectively. In (b), the regime $\beta>\beta_H$ ($\beta=1.8$): the system has undergone through a Hopf bifurcation but the stable limit cycle spreading from the origin is not present here, due to the collapse of the unstable cycle in the origin, leaving the outer orbit to become globally attractive. The red-colored and blue-colored solutions have initial conditions $\lambda(0)=0$, $y(0)=-0.005$ and $\lambda(0)=0$, $y(0)=6$ respectively. } \label{fig:casoB}
\end{figure}

\begin{itemize}
\item[-] For $\beta<\beta_1(\alpha)$, the origin is a global attractor. As is Case~A the value $\beta_1(\alpha)$ is strictly greater than the value in which the line first touches the graph $y=g'(x)$.
\item[-] For $\beta$~$\in$~$(\beta_1(\alpha), \beta_H)$ the origin is stable and we have an unstable periodic orbit contained in a stable periodic one. When $\beta$ increase the inner orbit shrinks and the outer expands.
\item[-] For $\beta=\beta_H$ the Hopf bifurcation is such that the origin looses stability, but this happens simultaneously to the collapse of the unstable periodic orbit on it. Therefore, after the bifurcation, we do not see the usual periodic orbit expanding form the origin because the unique orbit is the stable one (from the previous regime) that becomes globally stable.
\end{itemize}
This case is interesting because the Hopf bifurcation do not originates a small periodic orbit. However, the phenomenon is still a local one, because it is a small unstable orbit that collapses on the origin changing its stability.

\section{Proof of Theorem~\ref{THM:convergencegaussianmeasure2}}\label{Sec:proof_main_theorem}

As we mentioned before stating Theorem~\ref{THM:convergencegaussianmeasure2}, under Assumptions~\ref{assumption_gen} and~\ref{assumption_gauss}, the behavior of \eqref{nonlinear_limit_gaussian} is completely described by the system of \mbox{ODE}:
% n presence of a dissipative behavior for $\lambda_t$, the system of \mbox{ODE} that we study is the following:
\begin{equation}\label{ODE_diss_with_v}
\left\{
\begin{array}{l}
\dot{m_t}=\frac{\beta}{2}g'(\lambda_t)-\frac{m_t}{2\sigma_2},\\
\dot{v_t}=1-\frac{v_t}{\sigma^2},\\
\dot{\lambda_t}=-\alpha \lambda_t +\frac{\beta}{2} g'(\lambda_t)-\frac{m_t}{2\sigma^2}.\\
\end{array}
\right.
\end{equation}
By the independence of the evolution of $v_t$, we
study the two-dimensional system 
\begin{equation}\label{ODE_diss}
\left\{
\begin{array}{l}
\dot{m_t}=\frac{\beta}{2}g'(\lambda_t)-\frac{m_t}{2\sigma_2},\\
\dot{\lambda_t}=-\alpha \lambda_t +\frac{\beta}{2} g'(\lambda_t)-\frac{m_t}{2\sigma^2}.\\
\end{array}
\right.
\end{equation}
instead of the three-dimensional~\eqref{ODE_diss_with_v}. By the change of variable $y=\frac{1}{2\sigma^2}(\lambda-m)$ we see that \eqref{ODE_diss} is equivalent to the  Li\'enard system \eqref{ODE_lienard_change_variable} and we study its behavior.

\begin{theorem}\label{thm:dyn_system_with_diss}
Fix $\sigma^2>0$ and $\alpha>0$ and consider the dynamical system \eqref{ODE_lienard_change_variable} under Assumptions~\ref{assumption_gen} and~\ref{assumption_gauss}.
\begin{itemize}
\item[i)] There exists $\beta^*>0$ such that $\forall$~$\beta$~$\in$~$(0,\beta^*)$ the origin is a global attractor for \eqref{ODE_lienard_change_variable}.
\item[ii)] If $g''(0)>0$ { and $g'$ is not linear in 0}, the origin looses stability via a Hopf bifurcation at the critical value $\beta_H=\frac{2\alpha+\frac{1}{\sigma^2}}{g''(0)}.$
\item[iii)] If $g''(0)>0$ and there exists $C>0$ such that for all $x\in(C,\infty)$ the function $g'(x)$ is concave, then there exists a $\beta_{UC}$ such that for all $\beta>\beta_{UC}$ there exists a unique limit cycles for \eqref{ODE_lienard_change_variable}.
\end{itemize}
\end{theorem}
\begin{proof} \emph{i)} Let us consider the function
$$W(y,\lambda)=\frac{\alpha}{4\sigma^2}\lambda^2+\frac{y^2}{2},$$
it is clear that 
\begin{equation}\label{Lyapunov}\frac{d}{dt}W(y(t),\lambda(t))=-\frac{\alpha}{2\sigma^2}\lambda \left((\alpha+\frac{1}{2\sigma^2})\lambda-\frac{\beta}{2}g'(\lambda)\right).
\end{equation}
%As in the proof of Proposition~\ref{Prop:m_t}, 
We are left to consider the intersection of the graph of the function $y=g'(\lambda)$ with the line  $y=\frac{2\alpha+\frac{1}{\sigma^2}}{\beta}\lambda$. We see that there exists a $\beta^*$ sufficiently small, such that $\forall$ $\beta<\beta^*$ the only intersection is the origin, meaning that \eqref{Lyapunov} is strictly negative except than at $(0,0)$, in  which it is zero. Therefore $W$ is a global Lyapunov function, proving global attractivity of the origin.\\
\emph{ii)} We consider the system \eqref{ODE_lienard_change_variable} linearized around the point $(0,0)$, that gives the linear system:
$$
\left(\begin{array}{c}
\dot{y}\\
\dot{\lambda}
\end{array}\right)=\left(\begin{array}{ccc}
0&\text{ }&-\frac{\alpha}{2\sigma^2}\\
1&\text{ }&-\left(\alpha+\frac{1}{2\sigma^2}\right)+\frac{\beta}{2}g''(0)
\end{array}\right)\left(\begin{array}{c}
y\\
\lambda
\end{array}\right)$$
with eigenvalues
$$
x_{\pm}=\frac{1}{2}\left(\frac{\beta}{2}g''(0)-\alpha-\frac{1}{2\sigma^2}\pm\sqrt{\left(\frac{\beta}{2}g''(0)-\alpha-\frac{1}{2\sigma^2}\right)^2-\frac{2\alpha}{\sigma^2}}\right).$$
{When $\beta=\beta_H$, the eigenvalues $x_{\pm}$ constitute a pair of conjugate non-zero purely imaginary numbers, crossing the imaginary line with positive velocity. The assumption on $g'$ being not linear around 0 assures that we are in presence of a local bifurcation, and we can conclude that when $\beta=\beta_H$ we have a Hopf bifurcation.} \\
\emph{iii)} {Recall that under Assumption \ref{assumption_gen}, $g'$ is an odd function positive on $(0,\infty)$. Moreover, by Assumption \ref{assumption_gauss}, in the Gaussian case $g'$ has to be definitively sub-linear (it is clear by \eqref{integrabilita_g}). If, in addition, $g''(0)>0$ and there exists $C>0$ such that for all $x\in(C,\infty)$ the function $g'$ is concave, then there exists a $\beta_{UC}$ sufficiently large such that $\forall$~$\beta>\beta_{UC}$, the function $f_{\alpha,\beta}$ has exactly three zeros $-x^*<0<x^*$ and satisfies the following: $f_{\alpha,\beta}$ is negative on $(0,x^*)$ and positive and monotonically increasing on $(x^*,\infty)$.} In this way, for all $\beta>\beta_{UC}$, the system~\eqref{ODE_lienard_change_variable} satisfies the conditions for the existence and uniqueness of a limit cycle presented in Theorem~1.1 of~\cite{CaVi05}. The proof follows the usual approach for Li\'enard systems, used also in~\cite{DaFiRe13}.\\
First %it is shown that all the trajectories revolve around the origin. This is a consequence of the fact 
it is shown that the $y$ axis and the function $y=f_{\alpha,\beta}(\lambda)$ divides the $(\lambda,y)$-plane in four regions: 
\begin{align*}
I\doteq&\{(\lambda,y)\colon \lambda>0;y>f_{\alpha,\beta}(\lambda)\};\\
II\doteq&\{(\lambda,y)\colon \lambda>0;y<f_{\alpha,\beta}(\lambda)\};\\
III\doteq&\{(\lambda,y)\colon \lambda<0;y<f_{\alpha,\beta}(\lambda)\};\\
IV\doteq&\{(\lambda,y)\colon \lambda<0;y>f_{\alpha,\beta}(\lambda)\}.
\end{align*} 
In each of these four regions the vector field pushes the  trajectories to cross either the $y$ axis or the graph $y=f_{\alpha,\beta}(\lambda)$. Therefore each trajectory is forced to revolve clockwise around the origin.\\
Then, for $y_0>0$, we consider a trajectory starting from the point $(0,y_0)$ and we call $y_1>0$ its first intersection with the $y$-axis in the negative half-plane. We define the function 
$$
\Delta W(y_0)=W(0,y_1)-W(0,y_0).$$
Of course, when $\Delta W(\bar y)=0$, the trajectory starting from $(0,\bar y)$ is a periodic orbit. Let  $y^*_0>0$ be such that the trajectory starting at time zero from $(0,y^*_0)$ passes through $(x^*,0)$, the positive zero of the function $f_{\alpha,\beta}(\lambda)$. Then, let us define $t^*$ the time in which this trajectory crosses the $y$-axis, this time it must be $y(t^*)<0$, by the above observation. Notice that, with this notation and by~\eqref{Lyapunov_ODE}, we have that
\begin{equation}\label{W_int}
\Delta W(y^*)=\int_0^{t^*}\frac{d}{dt}W(y(t),\lambda(t))dt=-\int_{0}^{t^*}\frac{\alpha}{2\sigma^2}\lambda(t) f_{\alpha,\beta}(\lambda(t))dt.
\end{equation}
This trajectory is such that, for all $t\in[0,t^*]$, we have that $\lambda(t)~\in~[0,x^*]$, since in $(x^*,0)$ it moves from the region $I$ of the plane to the region $II$, in which the $\lambda$-coordinate starts to decrease. Moreover, it is clear that, for all $\lambda$~$\in$~$(0,x^*)$ the function $\lambda f_{\alpha,\beta}(\lambda)$ is strictly negative, due to the sign of $f_{\alpha,\beta}$, that is the reason why the integral in~\eqref{W_int} is strictly positive. Clearly, the same holds for all the trajectories closer than this one to the origin, i.e. $\Delta W(y)>0$ for all $y< y^*$. Then, since the function $f_{\alpha,\beta}(\lambda)$ remains positive for all $\lambda>x^*$, we see that  $\Delta W(y)$ decreases monotonically to $-\infty$, when $y\rightarrow\infty$, meaning that there exists a unique $\bar y$ for which it is zero.\\

%It is clear that there exists a $\beta_{UC}$ sufficiently large such that $\forall$~$\beta>\beta_{UC}$, the function $f_{\alpha,\beta}$ has exactly three zeros $-x^*<0<x^*$ and satisfies the following: $f_{\alpha,\beta}$ is negative on $(0,x^*)$ and positive and monotonically increasing on $(x^*,\infty)$. In this way, for all $\beta>\beta_{UC}$, the system~\eqref{ODE_lienard} satisfies the conditions for the existence and uniqueness of a limit cycle presented in Theorem~1.1 of~\cite{CaVi05}. The proof follows the usual approach for Li\'enard systems, used also in~\cite{DaFiRe13}. First it is shown that all the trajectories revolve around the origin. Then it is defined the function $$\Delta W(y_0)=W(0,y_1)-W(0,y_0),$$where $y_0$ and $y_1$ are the intersection of a trajectory with the $y$-axis, in the positive and the negative half-plane, respectively. Of course, when $\Delta W(y_0^*)=0$, the trajectory starting from $(0,y^*_0)$ is a limit cycle. Thus, it is proved that $\Delta W(y_0)>0$ for all $y_0$ sufficiently small and then it decreases monotonically to $-\infty$, when $y_0\rightarrow\infty$, meaning that there exists a unique $y_0^*$ for which it is zero.
\end{proof}

%and the number of cycles of the system depends on the positive zeros of the function $$G_{\alpha,\beta}(x)=\left(\alpha+\frac{1}{2\sigma^2}\right) x -\frac{\beta}{2} g'(x).$$  Notice that the number of stability points in the system without dissipation depends indeed on the zeros of $G_{0,\beta}(x)$.
Now, we aim to extend the results on the dynamical system \eqref{ODE_diss}, to the stochastic process itself.

\begin{proof}[Proof of Theorem~\ref{THM:convergencegaussianmeasure2}]

%\begin{theorem} \label{convergencegaussianmeasure} Fix $\beta>0$, 
%and let $\Lambda(\beta)$ as in \eqref{zero_points}, then the process $(X_t,\lambda_t)$ described by \eqref{nonlinear_limit_gaussian} has exactly $Card(\Lambda(\beta))$ stationary solution given by the measures$$\mu^*_{m}(dx,dl)=\frac{1}{\sqrt{2\pi \sigma^2}}e^{-\frac{(x-m)^2}{2\sigma^2}}dx\times\delta_m(dl)$$for all $m$~$\in$~$\Lambda(\beta)$. Moreover, for all $\mu_0(dx,d\lambda)=\nu_0(dx)\times \delta_{m_0}(d\lambda),$ square-integrable initial conditions with $m_0=\langle \mu_0,x\rangle$\begin{equation}\label{long-time}\lim_{t\rightarrow \infty}\|\mu_t(\cdot)-\mu^*_{m}(\cdot)\|_{TV}=0,\end{equation}where $m$ is the equilibrium point of \eqref{dm_t} such that $m_0$ belongs to the domain of attraction of $m$.
%\end{theorem}

%\begin{proof}%[Proof of Theorem \ref{convergencegaussianmeasure}]
%It is clear that the evolution given by \eqref{nonlinear_limit_gaussian} when $\alpha=0$ must have a law $\mu_t(dx,dl)=\nu_t(dx)\delta_{m_t}(dl)$ where $\delta_{m_t}$ is a Dirac delta centered in $m_t=\int_{\mathbb{R}}y\nu_t(dy)$.
{Let us restrict to the class of measures of the form $\mu(dx,dl)=\nu(dx)\times\delta_{\lambda}(dl)$, for $\nu \in \mathcal{M}(\mathbb{R})$ and $\lambda\in\mathbb{R}$. Suppose that there exist a stationary measure for \eqref{nonlinear_limit_gaussian} in this class, call it $\mu^*(dx,dl)=\nu^*(dx)\times\delta_{\lambda^*}(dl)$, then it must satisfy}
\begin{equation}\label{stat_fokker_plack}
\left\{\begin{array}{l}
0=\frac{1}{2}\frac{d^2}{d x^2}\nu^*(x)- \frac{d}{d x}\left[\left(\frac{\beta}{2}g'(\lambda^*)-\frac{x}{2\sigma^2}\right)\nu^*(x)\right]\\
0=-\alpha \lambda^* +\frac{\beta}{2} g'(\lambda^*)-\frac{m^*}{2\sigma^2}.
\end{array}\right.
\end{equation}
with $m^*=\int_{\mathbb{R}}x\nu^*(x)dx$. Let us solve the first row in \eqref{stat_fokker_plack}, we see that there exists $K$~$\in$~$\mathbb{R}$ such that
$$
\frac{d}{dx}\nu^*(x)=K+\left(\frac{\beta}{2}g'(\lambda^*)-\frac{x}{2\sigma^2}\right)\nu^*(x).
$$
Thus $\nu^*(x)$ solves a linear ODE, i.e. there exists $C$~$\in$~$\mathbb{R}$ such that
\begin{equation}
\nu^*(x)=\exp(\beta g'(\lambda^*)x-\frac{x^2}{2\sigma^2})\left(C+K\int_{\mathbb{R}}\exp(-\beta g'(\lambda^*)y+\frac{y^2}{2\sigma^2})dy\right).
\end{equation} 
Let us define the values of the constants:
\begin{itemize}
\item $K=0$, because if $K\neq0$ $\nu^*(x)$ is not integrable;
\item $C=\int_{\mathbb{R}}\exp(\beta g'(\lambda^*)x-\frac{x^2}{2\sigma^2})dx$, such that $\nu^*$ has mass one.
\end{itemize} 
The admissible functions $\nu^*$ are such that 
\[
m^*=\int_{\mathbb{R}}x\nu^*(x)dx=\beta\sigma^2 g'(\lambda^*),
\]
but they must simultaneously satisfy the second row in~\eqref{stat_fokker_plack}, then the only solution is such that $(m^*,\lambda^*)=(0,0)$ and this gives the thesis. \\

Now, let us prove the long-time behavior of $\mu_t$ for any square-integrable initial condition of the type $\mu_0=\nu_0\times\delta_{\lambda_0}$. As we said, this implies that $\mu_t=\nu_t\times\delta_{\lambda_t}$ %and $m_t=\mathbf{E}[X_t]$ 
for all $t>0$. %Finally, we state and prove a lemma concerning the  long-time behavior of the 
To this aim we introduce the following time-inhomogeneous SDE:
\begin{equation}
\label{linearSDE}
\begin{cases}
dY_t=\frac{\beta}{2}g'(\lambda_t)dt - {Y_t\over 2 \sigma^2}dt+dB_t\\
Y_0\in L^2(\Omega)
\end{cases}
\end{equation}
where $\lambda_t$ solves~\eqref{ODE_diss} with initial condition $(m_0,\lambda_0)=(\langle x,\nu_0\rangle,\nu_0\rangle,\lambda_0)$. The solution of \eqref{linearSDE} will be used as an auxiliary process to prove long-behavior of the solution of \eqref{nonlinear_limit_gaussian}. Indeed it is clear that, for all $t\geq0$, $X_t$ in \eqref{nonlinear_limit_gaussian} has the same law of $Y_t$ in \eqref{linearSDE}, provided that $\nu_0=Law(Y_0)$. In the following Lemma we prove that, if $(\langle x,\nu_0\rangle,\nu_0\rangle,\lambda_0)$ belongs to the domain of attraction of a certain closed subset $\gamma\in\mathbb{R}^2$ (being that an equilibrium point or a closed orbit), than the law of $Y_t$ approaches a gaussian law with mean $m^*$, where $(m^*,\lambda^*)\in \gamma$ for a certain $\lambda^*$.
 %and $a(t)=\frac{\beta g'(m_t)}{2}$ for all $t>0$, where $m_t$ solves  \eqref{dm_t} with $m_0=\mathbf{E}[X_0]$. The same argument holds true when $\alpha>0$, replacing $m_t$ with $\lambda_t$.

\begin{lemma}\label{lem:linearSDE}
Let $\{Y_t\}$ be the solution of \eqref{linearSDE}, let $P_t(Y_0,\cdot)$ be its law and let $(m_0,\lambda_0)$ belong to the domain of attraction of $\gamma$ according to the dynamics~\eqref{ODE_diss}. Then, 
$$
\lim_{t\to+\infty} \inf_{(a,b)\in\gamma}||P_t(Y_0,\cdot)-\nu_{a}(\cdot)||_{TV}=0
$$
where %$P_t(Y_0,\cdot)=Law(Y_t)$  and 
$\nu_{a}(dx)={1\over \sqrt{2\pi \sigma^2}}e^{(x-a)^2\over 2\sigma^2}dx.$

\end{lemma}
\begin{proof}[Proof of Lemma~\ref{lem:linearSDE}]
First of all, notice that
\begin{equation}
\label{limitebanale}
\lim_{t\to+\infty} \inf_{(a,b)\in\gamma}\left| \int_0^t e^{s-t\over 2\sigma^2} \frac{\beta}{2}g'(\lambda_s) ds-a\right|=0.
\end{equation}
Fix $t>0$, and $(a,b)\in\gamma$ then 
{\small\begin{align*}
 \int_0^t e^{s-t\over 2\sigma^2} \frac{\beta}{2}g'(\lambda_s) ds-a=& \int_0^t e^{s-t\over 2\sigma^2} \frac{\beta}{2}g'(\lambda_s) ds-e^{-\frac{t}{2\sigma^2}}\int_0^te^{\frac{s}{2\sigma^2}}\frac{m_s}{2\sigma^2}ds\\
 &+e^{-\frac{t}{2\sigma^2}}\int_0^te^{\frac{s}{2\sigma^2}}\frac{m_s}{2\sigma^2}ds-a
%\left(e^{-\frac{t}{2\sigma^2}}\int_0^te^{\frac{s}{2\sigma^2}}\frac{1}{2\sigma^2}ds+e^{-\frac{t}{2\sigma^2}}\right)
\\
 =&\int_0^t e^{s-t\over 2\sigma^2}\dot{m}_s ds+e^{-\frac{t}{2\sigma^2}}\int_0^te^{\frac{s}{2\sigma^2}}\frac{m_s}{2\sigma^2}ds-a\\
 =&m_t-a-e^{-\frac{t}{2\sigma^2}}m_0.
\end{align*}}
Since $(m_0,\lambda_0)$ belongs to the domain of attraction of $\gamma$, this implies~\eqref{limitebanale}.
Then, if $\mu_0(\cdot)=Law(Y_0)$,
\begin{multline*}\inf_{(a,b)\in\gamma}||P_t(Y_0,\cdot)-\nu_{a}(\cdot)||_{TV}=\inf_{(a,b)\in\gamma}\int_{\mathbb{R}} ||P_t(y,\cdot)-\nu_{a}(\cdot)||_{TV} d\nu_0( y)\\
= \inf_{(a,b)\in\gamma}{1\over 2}\int_{\mathbb{R}}\int_{\mathbb{R}} \left| {\exp\left({\left( x-ye^{-{t\over 2\sigma^2}} -\int_0^t e^{s-t\over 2\sigma^2} \frac{\beta}{2}g'(\lambda_s) ds  \right)^2\over 2\sigma^2(1-e^{-{t\over 2\sigma^2}})}\right)\over \sqrt{2\pi\sigma^2(1-e^{-{t\over 2\sigma^2}})}}  -  {e^{(x-a)^2\over 2\sigma^2}\over \sqrt{2\pi \sigma^2}} \right|dx d\nu_0(y)
\end{multline*} 
which converges to $0$ as $t\to+\infty$ thanks to \eqref{limitebanale}. 
\end{proof}
As we said, if $\nu_0=Law(Y_0)$,  $X_t$ in \eqref{nonlinear_limit_gaussian} has the same law of $Y_t$ in \eqref{linearSDE} for all $t\geq0$.  Then \eqref{long-time_origin_stable} and \eqref{long-time2} follows directly from Theorem~\ref{thm:dyn_system_with_diss} and Lemma~\ref{lem:linearSDE}.
\end{proof}

\section{Proof of the McKean-Vlasov limit}\label{SEC:prop_chaos}
In this section we collect the standard proofs of well-posedness of the nonlinear Markov process~\eqref{nonlinear_limit} and we identify its law as the weak limit of the sequence of empirical measures of~\eqref{lambda}.
Indeed, in \eqref{lambda} we describe the time-evolution of a stochastic process $(X^N(t),\lambda^N(t))_{t\geq0}$ with values in $\mathbb{R}^{2N}$ where, at every time $t\geq0$,  $X^N(t)=\left(X^N_1(t),\dots,X^N_N(t)\right)$ is the vector of the spins of the $N$ particles and
 $\lambda^N= \left(\lambda^N_1(t),\dots,\lambda^N_N(t)\right)$ is the vector of the ``perceived  magnetizations''. %of each particle.
 The Markov process $(X^N,\lambda^N)$ has infinitesimal generator 
\small \begin{align}\label{inf_gen_N}
\mathcal{L}^Nf(x,\lambda)&=\sum_{i=1}^N\left[ \frac{1}{2}\left(\beta g'(\lambda_i)+\frac{\rho'(x_i)}{\rho(x_i)}\right)\frac{\partial}{\partial x_i}f(x,\lambda)+\frac{1}{2}\frac{\partial^2}{\partial x_i^2}f(x,\lambda)\right.\\\nonumber
&+\frac{1}{2N}\sum_{j=1}^N\frac{\partial^2}{\partial x_i\partial \lambda_j}f(x,\lambda)+\left(\frac{1}{2N}\sum_{j=1}^N\left(\beta g'(\lambda_j)+\frac{\rho'(x_j)}{\rho(x_j)}\right)-\alpha \lambda_i\right)\frac{\partial}{\partial \lambda_i}f(x,\lambda)\\\nonumber
&\left.+\frac{D}{2}\frac{\partial^2}{\partial \lambda_i^2}f(x,\lambda)+\frac{1}{2N}\sum_{j=1}^N\frac{\partial^2}{\partial x_j\partial \lambda_i}f(x,\lambda)+\frac{1}{2N}\sum_{j=1}^N\frac{\partial^2}{\partial \lambda_j\partial \lambda_i}f(x,\lambda)\right].
\end{align}
%\mathcal{L}^Nf(x,\lambda)=\sum_{i=1}^N\left[ \frac{1}{2}\left(\beta g'(\lambda_i){+}\frac{\rho'(x_i)}{\rho(x_i)}\right)\frac{\partial}{\partial x_i}f(x,\lambda)+\frac{1}{2}\frac{\partial^2}{\partial x_i^2}f(x,\lambda)\right.\\\left.+\left(\frac{1}{2N}\sum_{j=1}^N\left(\beta g'(\lambda_j){+}\frac{\rho'(x_j)}{\rho(x_j)}\right)-\alpha \lambda_i\right)\frac{\partial}{\partial \lambda_i}f(x,\lambda)+\frac{D}{2}\frac{\partial^2}{\partial \lambda_i^2}f(x,\lambda)\right].\end{multline}
%\subsection*
We associate to~\eqref{lambda} the nonlinear Markov process described in~\eqref{nonlinear_limit}, which is commonly called the McKean-Vlasov limit of the particle system~\eqref{lambda}, see \cite{Szn91} and reference therein. In the following theorems we prove well-posedness of the nonlinear process and the so-called property of \emph{propagation of chaos} of the particle system~\eqref{lambda}.
\begin{theorem}\label{existenceuniqueness}
The nonlinear process \eqref{nonlinear_limit} is well-defined, i.e. there exists a unique strong solution for all square-integrable initial condition $(X_0,\lambda_0)$ $\in$ $\mathbb{R}^2$.
\end{theorem}

\begin{theorem}\label{Thm:prop_chaos} Let $(X^N(t),\lambda^N(t))_{t\geq0}$ be the Markov process with generator \eqref{inf_gen_N} starting from i.i.d. initial conditions $Law((X^N_i,\lambda^N_i))=\mu_0$ on $\mathbb{R}^2$, where $\int_{\mathbb{R}^2}(x^2+\lambda^2)\mu_0(dx,d\lambda)<\infty$, and denote with $P^N$ its law on $\mathbf{C}([0,T],\mathbb{R}^{2N})$. Let $(X(t),\lambda(t))_{t\geq0}$ be the solution to \eqref{nonlinear_limit} with initial condition $\mu_0$, and denote with $\mu$ its law on $\mathbf{C}([0,T],\mathbb{R}^2)$. Then, the sequence $(P^N)_{N\in \mathbb{N}}$ is $\mu$-chaotic.
\end{theorem}
The proofs of Theorem~\ref{existenceuniqueness} and~\ref{Thm:prop_chaos} follow a standard approach via pathwise estimates. To this aim, let us introduce a suitable distance on the space of trajectories. We fix a time $T>0$ and we define the $W_2$ Wasserstein distance on the set $\mathcal{M}^2(\mathbf{C}([0,T],\mathbb{R}^2))$ of square-integrable measures: for all $\mu,\nu$~$\in$~$\mathcal{M}^2(\mathbf{C}([0,T],\mathbb{R}^2))$ 
{\small \begin{multline}
D_{2,T}(\mu,\nu)^2=\inf\left\{\int \sup_{t\in[0,T]}\|x(s)-y(s)\|^2m(dx,dy), \right.\\
\left.\text{ with }m\in\mathcal{M}^2(\mathbf{C}([0,T],\mathbb{R}^2)\times \mathbf{C}([0,T],\mathbb{R}^2)), \pi_1\circ m=\mu,\, \pi_2\circ m=\nu \right\}.
\end{multline}  }

\begin{proof}[Proof of Theorem~\ref{existenceuniqueness}]
Given any square-integrable law $\mu_0$ on $\mathbb{R}^2$, we define a map $\Gamma$ that associates to a measure $Q$ $\in$ $\mathcal{M}^2(\mathbf{C}([0,T],\mathbb{R}^2))$ the law of the solution $\{(X_t,\lambda_t)\}_{t\in[0,T]}$ of the \mbox{SDE}
\begin{equation}\label{map}
\left\{\begin{array}{l}
dX_t=\frac{\beta}{2} g'(\lambda_t)dt+\frac{\rho'(X_t)}{2\rho(X_t)}dt+dB^1_t\\
d\lambda_t=-\alpha \lambda_t dt+\langle Q_t(dx,dl),\frac{\beta}{2} g'(l){+}\frac{\rho'(x)}{2\rho(x)} \rangle dt+DdB^2_t,
\end{array}
\right.
\end{equation}
that, for $\mu_0$ initial condition, admits a unique strong solution for classical results, see \cite{IkWa14}; of course a solution to \eqref{nonlinear_limit} is a fixed point of $\Gamma$. We use a coupling argument to prove existence (via a Picard iteration) and uniqueness of the fixed point of $\Gamma$. Let us start with the proof of uniqueness, if $Q^1$ and $Q^2$ are two fixed point of $\Gamma$, i.e. two measures in $\mathcal{M}^2(\mathbf{C}([0,T],\mathbb{R}^2))$ such that $Q^1=\Gamma(Q^1)$ and $Q^2=\Gamma(Q^2)$. We couple them as follows. Let $(\Omega,\mathcal{F},\{\mathcal{F}_t\}_{t\in[0,T]},\mathbf{P})$ be a filtered probability space and $\{B_t\}_{t\in[0,T]}$ a two-dimensional Brownian motion. Then we write 
\begin{equation}\label{X1L1}
\left\{\begin{array}{l}
dX^1_t=\frac{\beta}{2} g'(\lambda^1_t)dt{+}\frac{\rho'(X^1_t)}{2\rho(X^1_t)}dt+dB^1_t\\
d\lambda^1_t=-\alpha \lambda^1_t dt+\langle Q^1_t(dx,dl),\frac{\beta}{2} g'(l){+}\frac{\rho'(x)}{2\rho(x)} \rangle dt+DdB^2_t,
\end{array}
\right.
\end{equation}
and \begin{equation}\label{X2L2}
\left\{\begin{array}{l}
dX^2_t=\frac{\beta}{2} g'(\lambda^2_t)dt{+}\frac{\rho'(X^2_t)}{2\rho(X^2_t)}dt+dB^1_t\\
d\lambda^2_t=-\alpha \lambda^2_t dt+\langle Q^2_t(dx,dl),\frac{\beta}{2} g'(l){+}\frac{\rho'(x)}{2\rho(x)} \rangle dt+DdB^2_t,
\end{array}
\right.
\end{equation}
where the initial conditions are $(X^1_0,\lambda^1_0)=(X^2_0,\lambda^2_0)$ a.s., $\mu_0$ distributed. We estimate the distance between $Q^1$ and $Q^2$ by means of the above coupling, i.e.
$$
D_{2,T}(Q^1,Q^2)\leq\sqrt{\mathbf{E}\left[ \sup_{t\in[0,T]}(X^1_t-X^2_t)^2+(\lambda^1_t-\lambda^2_t)^2\right]}.$$
The \mbox{SDE} for $\lambda^1$ and $\lambda^2$ is linear, then we write explicitly
$$
\lambda^1_t-\lambda^2_t=\int_0^te^{\alpha(s-t)} \langle Q^1_s(dx,dl)-Q^2_s(dx,dl),\frac{\beta}{2} g'(l){+}\frac{\rho'(x)}{2\rho(x)} \rangle ds, 
$$
but notice that $\langle Q^1_t(dx,dl)-Q^2_t(dx,dl),\frac{\beta}{2} g'(l){+}\frac{\rho'(x)}{2\rho(x)} \rangle=\frac{d}{dt}\mathbf{E}\left[X^1_t-X^2_t\right]$, that gives
$$
\lambda^1_t-\lambda^2_t=\mathbf{E}\left[X^1_t-X^2_t\right]-\alpha\int_0^t\mathbf{E}\left[X^1_s-X^2_s\right]e^{-\alpha(t-s)}ds.
$$
We use Ito's formula to obtain
$$
(X^1_t-X^2_t)^2=2\int_0^t(X^1_s-X^2_s)\left( \frac{\beta}{2}g'(\lambda^1_s)- \frac{\beta}{2}g'(\lambda^1_s){+}\frac{\rho'(X^1_s)}{2\rho(X^1_s)}{-}\frac{\rho'(X^2_s)}{2\rho(X^2_s)}\right) ds.
$$
Therefore, there exists $C_T$ such that
{\small
$$
\mathbf{E}\left[ \sup_{t\in[0,T]}(X^1_t-X^2_t)^2+(\lambda^1_t-\lambda^2_t)^2\right]\leq C_T\int_0^T\mathbf{E}\left[ \sup_{t\in[0,s]}(X^1_t-X^2_t)^2+(\lambda^1_t-\lambda^2_t)^2\right]ds,
$$}
and by Gronwall Lemma this gives $D_{2,T}(Q^1,Q^2)=0$. With a Picard iteration of the type $Q^{n}=\Gamma(Q^{n-1})$ and with the above arguments, we obtain that 
\begin{multline}
\mathbf{E}\left[ \sup_{t\in[0,T]}(X^{n}_t-X^{n-1}_t)^2+(\lambda^n_t-\lambda^{n-1}_t)^2\right]\\
\leq L\int_0^T\mathbf{E}\left[ \sup_{t\in[0,s]}(X^n_t-X^{n-1}_t)^2+(\lambda^n_t-\lambda^{n-1}_t)^2\right]ds\\
+\alpha^2 T\int_0^TD_{2,s}(Q^{n-1},Q^{n-2})^2ds,
\end{multline}
that gives $D_{2,T}(Q^{n},Q^{n-1})^2\leq \frac{(e^{LT}T\alpha^2)^n}{n!}\int_0^TD_{2,s}(Q^{1},Q^{0})^2ds$, i.e. $\{Q^n\}_{n\in\mathbb{N}}$ is a Cauchy sequence for $D_{2,T}$ and therefore for a weaker, but complete, metric on $\mathcal{M}^2(\mathbf{C}([0,T],\mathbb{R}^2))$.

\end{proof}

\begin{proof}[Proof of Theorem~\ref{Thm:prop_chaos}] By the exchangeability of the components and of the dynamics, it is well-known that the thesis is implied by 
$
\mathbf{E}\left[D_{2,T}(\mu^N,\mu)\right]{\longrightarrow} 0
$
as $N\to+\infty$, where $\mu^N={1\over N}\sum_{i=1}^N \delta_{(X^N,\lambda^N)}$ (see \cite{Szn91}). Therefore, we apply this approach to prove the theorem. Fix a filtered probability space $(\Omega,\mathcal{F},\{\mathcal{F}_t\}_{t\in[0,T]},\mathbf{P})$, for any $N\in\mathbb{N}$, take a $2N$-dimensional Brownian motion $\{B_t\}_{t\in[0,T]}$ and consider the coupled processes given by 
{\small \begin{equation}
\label{microscopic1} \left\{\begin{array}{l}
dX^{N,i}_t=\frac{\beta}{2} g'(\lambda^{N,i}_t)dt+\frac{\rho'(X^{N,i}_t)}{2\rho(X^{N,i}_t)}dt+dB^{1,i}_t\\
d\lambda^{N,i}_t=-\alpha \lambda^{N,i}_t dt+{1\over N}\sum_{j=1}^N \left(\frac{\beta}{2} g'(\lambda^{N,j}_t)+\frac{\rho'(X^{N,j}_t)}{2\rho(X^{N,j}_t)}\right) dt+DdB^{2,i}_t+\frac{1}{N}\sum_{j=1}^NdB^{1,j}_t,
\end{array}
\right.
\end{equation}}
$i=1,\dots,N$, and
\begin{equation*}
\label{microscopic2}\left\{\begin{array}{l}
d\bar{X}^{N,i}_t=\frac{\beta}{2} g'(\bar{\lambda}^{N,i}_t)dt+\frac{\rho'(\bar{X}^{N,i}_t)}{2\rho(\bar{X}^{N,i}_t)}dt+dB^{1,i}_t\\
d\bar{\lambda}^{N,i}_t=-\alpha \bar{\lambda}^{N,i}_t dt+\langle \mu_t(dx,dl),\frac{\beta}{2} g'(l)+\frac{\rho'(x)}{2\rho(x)} \rangle dt+DdB^{2,i}_t,
\end{array}
\right.
\end{equation*}
$i=1,\dots,N$, where the initial conditions are $(X^{N,i}_0,\lambda^{N,i}_0)=(\bar{X}^{N,i}_0,\bar{\lambda}^{N,i}_0)$ a.s., $\mu_0^{\otimes N}$ distributed.
 Similarly to the proof of Theorem \ref{existenceuniqueness}, we write explicitly, for all $i=1,\dots,N$
{\small $$
\lambda^{N,i}_t-\bar {\lambda}^{N,i}_t=\int_0^te^{\alpha(s-t)} \langle \mu^N_s(dx,dl)-\mu_s(dx,dl),\frac{\beta}{2} g'(l)+\frac{\rho'(x)}{2\rho(x)} \rangle ds+\frac{1}{N}\sum_{j=1}^N\int_0^te^{\alpha(s-t)}dB^{1,i}_s.
$$}
However, from~\eqref{microscopic1}, we see that, a.s. it holds
{\small $$
\int_0^te^{\alpha(s-t)} \langle \mu^N_s(dx,dl),\frac{\beta}{2} g'(l)+\frac{\rho'(x)}{2\rho(x)} \rangle ds=\int_0^te^{\alpha(s-t)} d\langle \mu^N_s(dx,dl), x \rangle -\frac{1}{N}\sum_{j=1}^N\int_0^te^{\alpha(s-t)}dB^{1,i}_s,
$$}
and, consequently, we write  
\begin{align*}
\lambda^{N,i}_t-\bar{\lambda}^{N,i}_t=&\frac{1}{N}\sum_{j=1}^NX^{N,j}_t-\mathbf{E}\left[\bar{X}^{N,i}_t\right]-e^{-\alpha t}\left(\frac{1}{N}\sum_{j=1}^NX^{N,j}_0-\mathbf{E}\left[\bar{X}^{N,i}_0\right]\right)\\
&-\alpha\int_0^t\left(\frac{1}{N}\sum_{j=1}^NX^{N,j}_s-\mathbf{E}\left[\bar{X}^{N,i}_s\right]\right)e^{-\alpha(t-s)}ds.
\end{align*}
Now  let $\bar{\mu}^N={1\over N} \sum_{j=1}^N \delta_{(\bar{X}^{N,j},\bar{\lambda}^{N,j})}$; it is easy to see that, as in the proof of Theorem \ref{existenceuniqueness}, it holds
\begin{align*}
\mathbf{E}\left[D_{2,T}(\mu^N,\bar{\mu}^N)^2\right]\leq& \mathbf{E}\left[ \sup_{t\in[0,T]}(X^{N,1}_t-\bar{X}^{N,1}_t)^2+(\lambda^{N,1}_t-\bar{\lambda}^{N,1}_t)^2\right]\\
\leq & L\int_0^T\mathbf{E}\left[ \sup_{t\in[0,s]}(X^{N,1}_t-\bar{X}^{N,1}_t)^2+(\lambda^{N,1}_t-\bar{\lambda}^{N,1}_t)^2\right]ds\\
&+\alpha^2 T \int_0^T\mathbf{E}\left[D_{2,s}(\mu^N,\mu)^2\right]ds,
\end{align*}
which, by an application of Gronwall's Lemma, implies that there exists $C_T>0$ such that
\begin{equation}\label{steppropagation}
\mathbf{E}\left[D_{2,T}(\mu^N,\bar{\mu}^N)^2\right]\leq C_T \int_0^T\mathbf{E}\left[D_{2,s}(\mu^N,\mu)^2\right]ds.
\end{equation}
 Moreover,  it is well known that $\mathbf{E}\left[D_{2,T}(\bar{\mu}^N,\mu)\right]\leq \beta(N)$ for some sequence $\beta(N)$ such that $\lim_{N\rightarrow \infty}\beta(N)=0$. Then, using \eqref{steppropagation}, we have 
\begin{align*}
\mathbf{E}\left[D_{2,T}(\mu^N,\mu)^2\right]\leq& \mathbf{E}\left[D_{2,T}(\mu^N,\bar{\mu}^N)^2\right]+\mathbf{E}\left[D_{2,T}(\bar{\mu}^N,\mu)^2\right]\\ \leq & C_T \int_0^T\mathbf{E}\left[D_{2,s}(\mu^N,\mu)^2\right]ds+  \beta(N)\leq K_T\beta(N) 
\end{align*}
for some $K_T>0$, which concludes the proof.

\end{proof}

\textit{Acknowledgement.}  The authors wish to thank Paolo~Dai~Pra for having suggested the idea of this work and for all the useful discussions on it and Marco Formentin for the suggestions and comments. This work was partially supported by the INdAM-GNAMPA Project 2017~``Collective periodic behavior in interacting particle systems''. L.~A. acknowledges the partial support by Centro Studi Levi Cases (Universit\`a di Padova).
\bibliographystyle{abbrv}
\bibliography{Bibliography}

\end{document}